\documentclass[11pt,a4paper]{article}

\usepackage{amsmath}
\usepackage{amsthm}
\usepackage{amssymb}
\usepackage{cain_arxiv}
\usepackage{graphicx}

\theoremstyle{definition}
\newtheorem{definition}{Definition}[section]
\newtheorem{example}[definition]{Example}
\newtheorem{question}[definition]{Question}
\newtheorem{remark}[definition]{Remark}

\theoremstyle{plain}
\newtheorem{proposition}[definition]{Proposition}
\newtheorem{lemma}[definition]{Lemma}
\newtheorem{theorem}[definition]{Theorem}

\numberwithin{equation}{section}

\def\fullref#1#2{%
    {#1\space\penalty 200\relax\ref{#2}}%
}

\newcommand{\defterm}[1]{\textit{#1}}

\newcommand{\fsa}[1]{\mathcal{#1}}

\newcommand{\cgen}[1]{#1^{\#}}

\newcommand{\pres}[2]{\left\langle #1\:|\:#2 \right\rangle}

\newcommand{\nset}{\mathbb{N}}
\newcommand{\zset}{\mathbb{Z}}

\newcommand{\emptyword}{\varepsilon}

\newcommand{\rel}[1]{\mathcal{#1}}
\newcommand{\elt}[1]{\overline{#1}}

\DeclareMathOperator{\id}{id}

\newcommand{\imderives}{\Rightarrow}
\newcommand{\derives}{\Rightarrow^*}

\newcommand{\Stab}{\mathrm{Stab}}

\newcommand{\rev}{\mathrm{rev}}

\def\rep#1{\underline{#1}}

\def\lex#1{#1_{\mathrm{Lex}}}
\def\slex#1{#1_{\mathrm{SLex}}}

\newcommand{\symgroup}[1]{\mathcal{S}_{#1}}

\def\dash{\unskip --- \ignorespaces}

\begin{document}


\title{Markov semigroups, monoids, and groups}






\author{Alan J. Cain \& Victor Maltcev}
\date{}

\thanks{The first author's research was funded by the European
  Regional Development Fund through the programme {\sc COMPETE} and by the
  Portuguese Government through the {\sc FCT} \dash Fundação para a Ciência
  e a Tecnologia under the project PEst-C/MAT/UI0144/2011 and through an
  {\sc FCT} Ci\^{e}ncia 2008 fellowship. The authors thank Rostislav
  Grigorchuk and Michael Stoll for supplying references and
  offprints. The observation in
  \fullref{Remark}{rem:growthratesunequal} arose during a conversation
  with Nik Ru\v{s}kuc.}

\maketitle

\address[AJC]{%
Centro de Matem\'{a}tica, Universidade do Porto, \\
Rua do Campo Alegre 687, 4169--007 Porto, Portugal
}
\email{%
ajcain@fc.up.pt
}
\webpage{%
www.fc.up.pt/pessoas/ajcain/
}
\address[VM]{%
School of Mathematics \& Statistics, University of St Andrews,\\
North Haugh, St Andrews, Fife KY16 9SS, United Kingdom
}
\email{%
victor@mcs.st-andrews.ac.uk
}

\begin{abstract}
A group is Markov if it admits a prefix-closed regular language of
unique representatives with respect to some generating set, and
strongly Markov if it admits such a language of unique minimal-length
representatives over every generating set.  This paper considers the
natural generalizations of these concepts to semigroups and
monoids. Two distinct potential generalizations to monoids are shown
to be equivalent. Various interesting examples are presented,
including an example of a non-Markov monoid that nevertheless admits a
regular language of unique representatives over any generating set. It
is shown that all finitely generated commutative semigroups are
strongly Markov, but that finitely generated subsemigroups of
virtually abelian or polycyclic groups need not be. Potential
connections with word-hyperbolic semigroups are investigated. A study
is made of the interaction of the classes of Markov and strongly
Markov semigroups with direct products, free products, and
finite-index subsemigroups and extensions. Several questions are
posed.
\end{abstract}





\section{Introduction}

The notion of Markov groups was introduced by Gromov in his seminal
paper on hyperbolic groups \cite[\S~5.2]{gromov_hyperbolic}, and
explored further by Ghys \& de~la Harpe~\cite{ghys_markov}. A group is
Markov if it admits a language of unique representatives, with respect
to some generating set, that can be described by a Markov grammar. In
this context, a Markov grammar is essentially a finite state automaton
with one initial state and every state being an accept state. The
connection with hyperbolic groups arises because every hyperbolic
group admits such a language of \emph{minimal-length} unique
representatives; such groups are said be strongly
Markov~\cite[Th\'{e}or\`{e}me~13]{ghys_markov}. Strongly Markov groups
have rational growth series with respect to any generating
set~\cite[Corollaire~14]{ghys_markov}.

The overarching aim of this paper is to begin to investigate the
natural generalization to semigroups of this notion of Markov
groups. A motivation for this is the fruitful generalization from
groups to semigroups of concepts involving automata and languages,
such as automatic structures (for groups, see \cite{epstein_wordproc},
for semigroups, \cite{campbell_autsg}), automatic presentations (see,
for example, \cite{oliver_autopresgroups,cort_apsg}), and automaton
semigroups (for groups, see the monograph~\cite{nekrashevych_ssg}, for
semigroups, see for example \cite{maltcev_cayley,silva_lamplighter}).

After recalling some necessary background definitions and results in
\fullref{\S}{sec:preliminaries}, the generalization of the definition
to monoids and semigroups is given in
\fullref{\S}{sec:definition}. The generalization to monoids is
immediate: a Markov monoid is a monoid admitting a language of unique
representatives described by a Markov grammar (again, essentially a
finite state automaton with a unique initial state and every state
being an accept state), which is equivalent to admitting a
prefix-closed regular language of unique representatives (see
\fullref{Proposition}{prop:markovlangauto} below). A monoid is
strongly Markov if it admits a prefix-closed language of unique
minimal-length representatives with respect to any generating
set. However, since the empty word is not in general a valid
representative for an element of a semigroup, generalizing the
definition to semigroups entails excluding the empty word from the
otherwise prefix-closed language of unique representatives. Thus there
are, for monoids, distinct notions of `Markov as a monoid' and `Markov
as a semigroup'; fortunately, the concepts turn out to be equivalent, as
proved in \fullref{\S}{sec:monoids}.

Some of the basic properties of Markov semigroups are explained in
\fullref{\S}{sec:properties}. An example of a non-Markov monoid that
nevertheless admits a regular (non-prefix-closed) language of unique
representatives with respect to any generating set is given in
\fullref{\S}{sec:regnotmarkov}. How certain rewriting systems
naturally give rise to Markov semigroups is shown in
\fullref{\S}{sec:rs}. That finitely generated commutative semigroups
are strongly Markov is shown in \fullref{\S}{sec:comm}. Next,
\fullref{\S}{sec:polycyclic} shows that finitely generated
subsemigroups of polycyclic or virtually abelian groups need not be
Markov, and discusses the importance of these
facts. \fullref{\S}{sec:examples} exhibits some other interesting
examples of Markov semigroups and some examples of non-Markov
semigroups.

Given the intimate connection between hyperbolic groups and Markov
groups discussed above, it is natural to look for a parallel between
semigroups that are word-hyperbolic in the sense of Duncan \& Gilman
\cite{duncan_hyperbolic} and Markov semigroups. However, as discussed
in \fullref{\S}{sec:hyperbolicity}, a word-hyperbolic semigroup need
not even admit a regular language of unique normal forms, let alone a
prefix-closed one.

\fullref{\S\S}{sec:adjoin}--\ref{sec:finind} examine the interaction
of Markov semigroups with adjoining identities and zeros, with direct
products, with free products, and with finite-index subsemigroups and
extensions. Finally, the class of languages that are Markov languages
for semigroups is considered in \fullref{\S}{sec:markovlang}.

Since Markov semigroups seem to be an entirely new area, there are
many possible directions for further research. Consequently, various
open questions are scattered throughout the paper in the relevant
contexts.

We remark that the research described in this paper has involved
drawing techniques, ideas, and examples from a broad swathe of
semigroup and formal language theory.

\section{Preliminaries}
\label{sec:preliminaries}

\subsection{Generators, alphabets, and words}

The notation used in this paper distinguishes a word from the element
of the semigroup or monoid it represents. Let $A$ be an alphabet
representing a set of generators for a semigroup or monoid
$S$. Formally, there is a map $\phi : A \to S$ that extends to a
surjective homomorphism $\phi : A^+ \to S$ (or $\phi : A^* \to S$ if
$S$ is a monoid).

While occasionally the representation map $\phi$ will be explicitly
mentioned, generally the following notational distinction will
suffice: for a word $w \in A^*$, denote by $\elt{w}$ the element of
$M$ represented by $w$ (so that $\elt{w} = w\phi$); for a set of words
$W \subseteq A^*$, denote by $\elt{W}$ the set of all elements of $S$
represented by at least one word in $W$. Notice that the emptyword
$\emptyword$ is a valid representative word if and only if $S$ is a
monoid.

\subsection{Languages and automata}

For background information on regular and context-free languages and
finite automata, see \cite[Ch.~2--4]{hopcroft_automata}.

Let $L$ be a language over an alphabet $A$. Then $L$ is \defterm{prefix-closed} if
\[
(\forall u \in A^*,v \in A^+)(uv \in L \implies u \in L),
\]
and $L$ is \defterm{closed under taking non-empty prefixes}, or more
succinctly \defterm{$+$-prefix-closed}, if
\[
(\forall u \in A^+,v \in A^+)(uv \in L \implies u \in L).
\]
Notice that if $L$ is prefix-closed and non-empty, it contains the
empty word $\emptyword$.


\subsection{String-rewriting systems}

This subsection contains facts about string rewriting needed later in
the paper. For further background information, see
\cite{book_srs}.

A \defterm{string rewriting system}, or simply a \defterm{rewriting
  system}, is a pair $(A,\rel{R})$, where $A$ is a finite alphabet and
$\rel{R}$ is a set of pairs $(\ell,r)$, known as \defterm{rewriting
  rules}, drawn from $A^* \times A^*$. The single reduction relation
$\imderives$ is defined as follows: $u \imderives v$ (where $u,v \in
A^*$) if there exists a rewriting rule $(\ell,r) \in \rel{R}$ and
words $x,y \in A^*$ such that $u = x\ell y$ and $v = xry$. That is, $u
\imderives v$ if one can obtain $v$ from $u$ by substituting the word
$r$ for a subword $\ell$ of $u$, where $(\ell,r)$ is a rewriting
rule. The reduction relation $\derives$ is the reflexive and
transitive closure of $\imderives$. The process of replacing a subword
$\ell$ by a word $r$, where $(\ell,r) \in \rel{R}$, is called
\defterm{reduction}, as is the iteration of this process.

A word $w \in A^*$ is \defterm{reducible} if it contains a subword $\ell$
that forms the left-hand side of a rewriting rule in $\rel{R}$; it is
otherwise called \defterm{irreducible}.

The string rewriting system $(A,\rel{R})$ is \defterm{noetherian} if
there is no infinite sequence $u_1,u_2,\ldots \in A^*$ such that $u_i
\imderives u_{i+1}$ for all $i \in \nset$. That is, $(A,\rel{R})$ is
noetherian if any process of reduction must eventually terminate with
an irreducible word. The rewriting system $(A,\rel{R})$ is
\defterm{confluent} if, for any words $u, u',u'' \in A^*$ with $u
\derives u'$ and $u \derives u''$, there exists a word $v \in A^*$
such that $u' \derives v$ and $u'' \derives v$. 

The string rewriting system $(A,\rel{R})$ is
\defterm{non-length-increasing} if $(\ell,r) \in \rel{R}$ implies
that $|\ell| \geq |r|$ and is \defterm{length-reducing} if $(\ell,r)
\in \rel{R}$ implies that $|\ell| > |r|$. Observe that any
length-reducing rewriting system is necessarily noetherian.

The rewriting system $(A,\rel{R})$ is \defterm{monadic} if it is
length-reducing and the right-hand side of each rule in $\rel{R}$ lies
in $A \cup \{\emptyword\}$; it is \defterm{special} if it is
length-reducing and each right-hand side is the empty word
$\emptyword$. Observe that every special rewriting system is also monadic.

The string rewriting system $(A,\rel{R})$ is \defterm{finite} if the
set of rules $\rel{R}$ is finite. A monadic rewriting system
$(A,\rel{R})$ is \defterm{regular} (respectively,
\defterm{context-free}), if, for each $a \in A \cup\{\emptyword\}$,
the set of all left-hand sides of rules in $\rel{R}$ with right-hand
side $a$ is regular (respectively, \defterm{context-free}).

Let $(A,\rel{R})$ be a confluent noetherian string rewriting
system. Then for any word $u \in A^*$, there is a unique irreducible
word $v \in A^*$ with $u \derives v$
\cite[Theorem~1.1.12]{book_srs}. The irreducible words are said to be
in \defterm{normal form}. The monoid presented by $\pres{A}{\rel{R}}$
may be identified with the set of normal form words under the
operation of `concatenation plus reduction to normal form'.

\section{Definitions}
\label{sec:definition}

As defined by Ghys \& de~la
Harpe~\cite[D\'{e}finition~4]{ghys_markov}, a group is Markov if it
admits a language of unique representatives defined by a
\defterm{Markov grammar}, which is essentially a finite state
automaton where every state is an accept state
\cite[D\'{e}finition~1]{ghys_markov}. The following result shows that
the class of languages recognized by such automata are the
prefix-closed regular languages. In general, arguments in this paper
work with regular expressions rather than explicitly constructed
automata, so this equivalences embodied in this result and in the
later \fullref{Proposition}{prop:semigroupmarkovlangauto} are
important.

\begin{proposition}
\label{prop:markovlangauto}
A regular language is prefix-closed if and only if it is recognized by
a finite state automaton in which every state is an accept state.
\end{proposition}

\begin{proof}
Suppose $L$ is prefix-closed and let $\fsa{A}$ be a trim deterministic
finite state automaton recognizing $L$. Let $q$ be some state of
$\fsa{A}$. Since $\fsa{A}$ is trim, $q$ lies on a path from the
initial state to an accept state. Let $w$ be the label on such a path,
with $w'$ being the label before the first visit to $q$. Then $w'$,
being a prefix of $w$, also lies in $L$. Since $\fsa{A}$ is
deterministic, there is only one path starting at the initial state
labelled by $w'$, and this path ends at $q$. Since $w' \in L$, it
follows that $q$ is an accept state. Therefore, since $q$ was
arbitrary, every state of $\fsa{A}$ is an accept state.

Suppose that $L$ is accepted by an automaton $\fsa{A}$ in which every
state is an accept state. Let $w \in L$ and let $w'$ be some prefix of
$w$. Then $w$ labels a path starting at the initial state of $\fsa{A}$
and leading to an accept state. The prefix $w'$ labels an initial
segment of this path, ending at a state $q$, which, by hypothesis, is
also an accept state. Thus $w' \in L$. Since $w \in L$ was arbitrary,
$L$ is prefix-closed.
\end{proof}

In light of \fullref{Proposition}{prop:markovlangauto}, a group is
Markov if it admits a prefix-closed regular language of unique
representatives. Now, in generalizing the notion of being Markov from
groups to semigroups, one must change from monoid to semigroup
generating sets and modify the notion of the language of
representatives appropriately. For groups, the language of
representatives is taken over an alphabet representing a monoid
generating set for the group, with the empty word being the
representative of the identity. (Indeed, the empty word lies in any
non-empty prefix-closed language.) In generalizing to arbitrary
semigroups, it is necessary to use a semigroup generating set, in
which case the empty word is no longer admissable as a representative,
and the natural definition for the language of representatives
requires not prefix-closure, but only $+$-prefix-closure.

This raises a potential problem, in that a monoid (possibly a group)
could be Markov in two different ways: it could be Markov as a monoid
(allowing, or rather requiring, that the identity be represented by
the empty word), or Markov as a semigroup (requiring that the identity
be represented by a non-empty word). It is thus conceivable that the
class of monoids that are Markov as monoids and the class of monoids
that are Markov as semigroups are distinct. Fortunately, however, the
two notions are equivalent, as will be shown in
\fullref{\S}{sec:monoids}.

The definition of `Markov as a monoid' is given first, since it is the
more direct generalization from the group case:

\begin{definition}
Let $M$ be a monoid and let $A$ be a finite alphabet representing a
monoid generating set for $M$. For $x \in M$, let $\lambda_A(x)$ be the
length of the shortest word over $A$ representing $x$; this is called
the natural length of $x$. (Notice that $\lambda(1_M) = 0$.)

A \defterm{monoid Markov language} for $M$ over $A$ is a regular
language $L$ that is prefix-closed and contains a unique
representative for every element of $M$.

A \defterm{robust monoid Markov language} for $M$ over $A$ is a
regular language $L$ that is prefix-closed and contains a unique
representative for every element of $M$ such that $|w| =
\lambda_A(\elt{w})$ for every $w \in L$.

The monoid $M$ is \defterm{Markov} (as a monoid) if there exists a
monoid Markov language for $M$ over an alphabet representing some
monoid generating set for $M$.

The monoid $M$ is \defterm{robustly Markov} (as a monoid) with respect
to an alphabet $A$ representing a generating set for $M$ if there
exists a robust monoid Markov language for $M$ over $A$.

The monoid $M$ is \defterm{strongly Markov} (as a monoid) if, for
every alphabet $A$ representing a monoid generating set for $M$, there
exists a robust monoid Markov language for $M$ over $A$.
\end{definition}

The reason for introducing the term `robustly Markov' is because there
are many natural examples of semigroups that admit a Markov languages
of minimal-length representatives while not being strongly Markov (see
for example \fullref{Proposition}{prop:rsmarkov}), and consequently
such semigroups still enjoy certain pleasant properties.

Note that Ghys \& de la Harpe \cite{ghys_markov} use different
terminology: rather than `Markov (respectively, strongly Markov)
groups', they use (terms that translate as) `groups with the Markov
(respectively, strong Markov) property'. We prefer Gromov's original
terminology, since it does not clash with `Markov property' in the
sense of an undecidable semigroup-theoretic property
(see~\cite{markov_impossibility} and \cite[Theorem~7.3.7]{book_srs}).

\begin{definition}
Let $S$ be a semigroup and let $A$ be a finite alphabet representing a
generating set for $S$. For $x \in S$, let $\lambda_A(x)$ be the length
of the shortest non-empty word over $A$ representing $x$; this is
called the natural length of $x$. (Notice that if $S$ is a monoid,
$\lambda_A(1_S)$ is not zero.)

A \defterm{semigroup Markov language} for $S$ over $A$ is a regular
language $L$ that does not contain the empty word, is
$+$-prefix-closed, and contains a unique representative for every
element of $S$.

A \defterm{robust semigroup Markov language} for $S$ over $A$ is a
regular language $L$ that does not contain the empty word, is
$+$-prefix-closed, and contains a unique representative for every
element of $S$ such that $|w| = \lambda_A(\elt{w})$.

The semigroup $S$ is \defterm{Markov} (as a semigroup) if there exists a
semigroup Markov language for $S$ over an alphabet representing some
generating set for $S$.

The semigroup $S$ is \defterm{robustly Markov} (as a semigroup) with respect
to an alphabet $A$ representing a generating set for $S$ if there
exists a robust semigroup Markov language for $S$ over $A$.

The semigroup $S$ is \defterm{strongly Markov} (as a semigroup) if, for
every alphabet $A$ representing a generating set for $S$, there
exists a robust semigroup Markov language for $S$ over $A$.
\end{definition}

The following result is the parallel of
\fullref{Proposition}{prop:semigroupmarkovlangauto} that applies to
$+$-prefix-closed languages:

\begin{proposition}
\label{prop:semigroupmarkovlangauto}
A regular language that does not contain the empty word is
$+$-prefix-closed if and only if it is recognized by a finite state
automaton in which every state except the initial state is an accept
state, and in which there are no incoming edges to the initial state.
\end{proposition}

\begin{proof}
Suppose $L$ is $+$-prefix-closed and does not contain the empty
word. Let $\fsa{A}$ be a trim deterministic finite state automaton
recognizing $L$. Since $L$ does not contain the empty word, the
initial state $q_0$ is not an accept state. Let $q$ be some other
state of $\fsa{A}$. Since $\fsa{A}$ is trim, $q$ lies on a path from
the initial state to an accept state. Let $w$ be the label on such a
path, with $w' \neq \emptyword$ being the label before the first visit
to $q$. Then $w'$, being a non-empty prefix of $w$, also lies in
$L$. Since $\fsa{A}$ is deterministic, there is only one path starting
at the initial state labelled by $w'$, and this path ends at
$q$. Since $w' \in L$, it follows that $q$ is an accept
state. Therefore, since $q$ was arbitrary, every state of $\fsa{A}$ is
an accept state. Finally, suppose, with the aim of obtaining a
contradiction, that there is an incoming edge from a state $p$ to the
initial state $q_0$. Then, since $\fsa{A}$ is trim, there is a word
$w$ labelling a path from $q_0$ to an accept state, including this
edge from $p$ to $q_0$. Let $w'$ be the prefix of $w$ labelling the
non-empty initial segment of the path from $q_0$ back to $q_0$. Then,
since $q_0$ is not an accept state and $\fsa{A}$ is deterministic, $w'
\notin L$, contradicting the fact that $L$ is $+$-prefix-closed. Hence
there are no edges ending at $q_0$.

Suppose that $L$ is accepted by an automaton $\fsa{A}$ in which every
state except the initial state is an accept state, and in which the
initial state has no incoming edges. Let $w \in L$ and let $w'$ be
some prefix of $w$. Then $w$ labels a path starting at the initial
state of $\fsa{A}$ and leading to an accept state. The prefix $w'$
labels an initial segment of this path, ending at a state $q$, which
cannot be the initial state, since it has no incoming edges, and must
therefore, by hypothesis, be an accept state. Thus $w' \in L$. Since
$w \in L$ was arbitrary, $L$ is prefix-closed.\qedhere
\end{proof}

\section{Markov monoids}
\label{sec:monoids}

As remarked in \fullref{\S}{sec:definition}, it is conceivable that
the class of monoids that are Markov as monoids and the class of
monoids that are Markov as semigroups are distinct, and the same issue
arises for being robustly Markov and strongly Markov. Fortunately, for
monoids the monoid and semigroup notions are equivalent, as the
following three results show:

\begin{proposition}
\label{prop:markovsgmon}
A monoid is Markov as a semigroup if and only if it is Markov as a monoid.
\end{proposition}

\begin{proof}
Let $M$ be a monoid.

Suppose that $M$ is Markov as a monoid. Let $A$ be an alphabet
representing a monoid generating set for $M$ such that there is a
monoid Markov language $L$ for $M$ over $A$. Then $L$ is
prefix-closed, regular, and contains a unique representative for each
element of $M$. In particular, the identity of $M$ is represented by
$\emptyword \in L$. Let $1$ be a new symbol representing the identity
for $M$. Then $K = (L - \{\emptyword\}) \cup \{1\}$ is
$+$-prefix-closed, regular, and contains a unique representative for
every element of $M$. Hence $K$ is a semigroup Markov language for $M$
and thus $M$ is Markov as a semigroup.

\smallskip
Suppose now that $M$ is Markov as a semigroup. Let $A$ be an alphabet
representing a semigroup generating set for $M$ such that there is a
semigroup Markov language $L$ for $M$ over $A$. Then $L$ is
$+$-prefix-closed, regular, and contains a unique representative for
every element of $M$. Let $w$ be the unique word in $L$ representing
the identity of $M$. Let
\[
K = \big(L - wA^*\big) \cup \{u \in A^* : wu \in L\}.
\]
Since $L$ is $+$-prefix-closed and $wA^*$ is closed under
concatenation on the right, $L - wA^*$ is also $+$-prefix
closed. Furthermore, $\{u \in A^* : wu \in L\}$ is
prefix-closed. (Notice that this set contains $\emptyword$ since $w$
lies in $L$.) So $K$ is prefix-closed. Moreover, $wu$ and $u$ represent
the same element of $M$ for any $u \in A^*$, so $\{u \in A^* : wu \in
L\}$ consists of unique representatives for exactly those elements
of $M$ whose representatives in $L$ have $w$ as a prefix. Hence every
element of $M$ has a unique representative in $K$. Finally, notice
that $K$ is regular. Thus $K$ is a monoid Markov language for $M$ and
so $M$ is Markov as a monoid.
\end{proof}

\begin{proposition}
\label{prop:robustlymarkovsgmon}
\begin{enumerate}
\item If a monoid is robustly Markov as a monoid with respect to some
  alphabet $A$ representing a semigroup generating set, it is also
  robustly Markov as a semigroup with respect to $A$.  Furthermore, if
  a monoid is robustly Markov as a monoid with respect to an alphabet
  $B$ representing a monoid generating set that is not also a
  semigroup generating set, then it is robustly Markov as a semigroup
  with respect to $B \cup \{1\}$, where $1$ represents the identity.
\item If a monoid is robustly Markov as a semigroup with respect to
  some alphabet $A$ representing a (semigroup) generating set, then it
  is robustly Markov as a monoid with respect to $A$. Furthermore, if
  a monoid is robustly Markov as a semigroup with respect to $B \cup
  \{1\}$, where $B$ represents a monoid generating set and $1$
  represents the identity, then it is robustly Markov as a monoid with
  respect to $B$.
\end{enumerate}
\end{proposition}

\begin{proof}
Let $M$ be a monoid.
\begin{enumerate}
\item Suppose that $M$ admits a robust monoid Markov language $L$ over
$A$. Since $\elt{A}$ generates $M$ as a semigroup, one can choose a
shortest non-empty word $w$ over $A$ representating the identity of
$M$. Let $w = w_1\cdots w_n$, with $w_i \in A$. For each non-empty
prefix $w_1\cdots w_i$ of $w$, let $p_i$ be the unique element of $L$
representing the same element of $M$ as this prefix. Notice that if an
element of $L$ has a prefix representing $\elt{w_1\cdots w_i}$, that
prefix must be $p_i$ by the prefix-closure of $L$ and the fact that it
maps bijectively onto $M$. Moreover, the length of $p_i$ must be the
same as the length of $w_1\cdots w_i$. To find a robust semigroup
Markov language for $M$ over $A$, it is necessary to replace the
prefixes $p_i$ by $w_1\cdots w_i$ and the empty word $\emptyword$ by
$w$. More formally, let
\[
K = \Big(\big(L - \{\emptyword\}\big) - \bigcup_{i=1}^n p_iA^*\Big) \cup \{w\} \cup \bigcup_{i=1}^n \{w_1\cdots w_i u : p_iu \in L\}.
\]
Now, $L - \{\emptyword\}$ is $+$-prefix-closed. Since each language
$p_iA^*$ is closed under concatenation on the right,
\[
\big(L - \{\emptyword\}\big) - \bigcup_{i=1}^n p_iA^*
\]
is $+$-prefix-closed. Furthermore,
\[
\{w\} \cup \bigcup_{i=1}^n \{w_1\cdots w_i u : p_iu \in L\}
\]
is $+$-prefix-closed since $L$ is and since every prefix of $w$ is in
this set. Therefore $K$ is $+$-prefix-closed. Furthermore, $K$ is
regular and, by definition, maps bijectively onto $M$. Finally, since
$|p_i| = |w_1\cdots w_i|$, it follows that the representative in $K$
of an element of $M$ is the same length as its representative in $L$,
excepting that the identity is represented by the non-empty word $w$
in $K$. So $K$ is a robust semigroup Markov language over
$A$ for $M$.

For the final claim, let $L$ be a robust monoid Markov language for
$M$ over $B$. Then $1$ is a shortest non-empty representative of $1_M$
over the alphabet $B \cup \{1\}$. Then $K = (L - \{\emptyword\}) \cup
\{1\}$ is a regular, $+$-prefix-closed, and consists of minimal-length
unique representatives for $M$. So $K$ is a robust semigroup Markov
language for $M$.

\item Suppose that $M$ admits a robust semigroup Markov language $L$ over an
alphabet $A$ representing a semigroup generating set for $M$. 

Let $w \in L$ be the representative of the identity of $M$. Since $L$
does not contain the empty word, $|w| \geq 1$. Suppose that some word
$u \in L$ contains $w$ as a proper subword, with $u = u'wu''$. Then
$\elt{u'u''} = \elt{u}$ and $|u'u''| < |u|$, which contradicts the
fact that representatives in $L$ are supposed to be length-minimal. So
$w$ is not a proper subword of any word in $L$. In particular, $L' = L
- \{w\}$ is $+$-prefix-closed.

Notice that $L'$ is $+$-prefix-closed, regular, and
consists of unique representatives having minimal length (over $A$)
for non-identity elements of $M$. Thus $K = L' \cup \{\emptyword\}$ is
prefix-closed, regular, and consists of unique representatives for all
elements of $M$. So $K$ is a robust monoid Markov language over $A$
for $M$.

For the final claim, let $A = B \cup \{1\}$ and follow the same
reasoning. In this case, $1$ is the minimal-length representative for
$1_M$ and does not occur as a subword of any other element of $L$. So
$L' \subseteq B^+$ and so $K$ is a robust monoid Markov language over $B$ for
$M$. \qedhere
\end{enumerate}
\end{proof}

The following result is a consequence of
\fullref{Proposition}{prop:robustlymarkovsgmon}:

\begin{proposition}
\label{prop:stronglymarkovsgmon}
A monoid is strongly Markov as a semigroup if and only if it is
strongly Markov as a monoid.
\end{proposition}

\begin{proof}
Let $M$ be a monoid.

Suppose $M$ is strongly Markov as a monoid. Let $A$ be an alphabet
representing a semigroup generating set for $M$. Then $M$ is robustly
Markov as a monoid with respect to $A$. By the first part of
\fullref{Proposition}{prop:robustlymarkovsgmon}, $M$ is robustly
Markov as a semigroup with respect to $A$. Since $A$ was an arbitrary
alphabet representing a semigroup generating set for $M$, by
definition $M$ is strongly Markov as a semigroup.

Suppose $M$ is strongly Markov as a semigroup. Let $B$ be an alphabet
representing a monoid generating set for $M$. Then $M$ is robustly
Markov as a semigroup with respect to $B \cup \{1\}$, where $\elt{1} =
1_M$. By the second part of
\fullref{Proposition}{prop:robustlymarkovsgmon}, $M$ is robustly
Markov as a monoid with respect to $B$. Since $B$ was an arbitrary
alphabet representing a monoid generating set for $M$, by definition
$M$ is strongly Markov as a monoid.
\end{proof}

In light of \fullref{Propositions}{prop:markovsgmon},
\ref{prop:robustlymarkovsgmon}, and \ref{prop:stronglymarkovsgmon}, there is no
need for a terminological distinction between the conditions `Markov
as a semigroup' and `Markov as a monoid', between `robustly Markov as
a semigroup' and `robustly Markov as a monoid', and between `strongly
Markov as a semigroup' and `strongly Markov as a monoid': the terms
`Markov', `robustly Markov', and `strongly Markov' alone will suffice.

The results in this section parallel the situation for automatic
monoids: a monoid is automatic as a semigroup if and only if it is
automatic as a monoid~\cite[\S5]{duncan_change}.

\section{Basic properties}
\label{sec:properties}

It is important to note that a Markov language does not define a group
or semigroup up to isomorphism, unlike an automatic structure
\cite[Proposition~2.3]{kambites_decision}. To see this, notice that if
$A$ is a finite alphabet of size $n$, then $A$ (qua language of
one-letter words) is a semigroup Markov language for any semigroup of
size $n$, and $A \cup \{\emptyword\}$ is a monoid Markov language for
any monoid or group of size $n+1$. The language $(a^* \cup
(a^{-1})^*)(b^* \cup (b^{-1})^*)(c^* \cup (c^{-1})^*)$ is a Markov
language for both $\zset^3$ and the Heisenberg group
\cite[\S~5.2]{ghys_groupeshyperboliquessurvey}.

The \defterm{growth series} of a semigroup $S$ with respect to a finite alphabet
$A$ representing a generating set for $S$ is
\[
\Sigma(S,A) = \sum_{s \in S} x^{\lambda_A(x)},
\]
or equivalently
\[
\Sigma(S,A) = \sum_{n = 0}^\infty \sigma_A(n)x^n,
\]
where $\sigma_A(n) = |\{s \in S : \lambda_A(s) = n\}|$. A growth
series $\Sigma(S,A)$ is said to be \defterm{rational} if it is a power
series expansion of a rational function. 

\begin{theorem}
\label{thm:robustrational}
If a semigroup admits a robust Markov language with respect to a
particular generating set, then its growth series with respect to that
generating set is a rational function. A strongly Markov semigroup has rational
growth series with respect to any generating set.
\end{theorem}

\begin{proof}
The proof for groups generalizes directly~\cite[Corollaire~14]{ghys_markov}.
\end{proof}

The independent importance of semigroup growth series (see, for example,
\cite[\S~4]{grigorchuk_growthsurvey}) means that, as a consequence of
\fullref{Theorem}{thm:robustrational}, robust Markov semigroups
are of considerably greater interest than Markov semigroups
generally.

\begin{remark}
\label{rem:growthratesunequal}
It is worth observing that the growth rate of a Markov language need
not mirror the growth of the semigroup or monoid. For example, all
finitely generated polycyclic groups are Markov
\cite[Corollaire~11]{ghys_markov}. Furthermore, the language of
collected words for a finitely generated polycyclic group forms a
Markov language \cite[p.~395]{sims_computation} and is easily seen to
have polynomial growth. However, a polycyclic group that is not
virtually nilpotent contains a free subsemigroup of rank~$2$
\cite[Theorem~4.12]{rosenblatt_invariant} and hence has exponential
growth.
\end{remark}

Being Markov implies the existence of a regular language of unique
normal forms over any finite generating set:

\begin{proposition}
\label{prop:markovchangegen}
Let $S$ be a semigroup that admits a regular language of unique normal
forms over some generating set (such as a Markov semigroup), and let
$A$ be a finite alphabet representing a generating set for $S$. Then
there is a regular language $L$ over $A$ such that every element of $S$ has a
unique representative in $L$.
\end{proposition}

[Notice that even if $S$ is a Markov semigroup, the language $L$ need
  not be prefix-closed.]

\begin{proof}
Let $K$ be a regular language of unique normal forms for $S$ over some
finite alphabet $B$. For each $b \in B$, let $u_b \in A^+$ be such
that $u_b$ represents $\elt{b}$. Let $R \subseteq B^+ \times A^+$ be the
rational relation:
\[
R = \{(b_1,u_{b_1})(b_2,u_{b_2})\cdots(b_n,u_{b_n}) : b \in B, n \in \nset\}
\]
Notice that if $(v,w) \in R$, then $\elt{v} = \elt{w}$.

Let
\[
L = K \circ R = \big\{w \in A^* : (\exists v \in K)((v,w) \in R)\big\};
\]
observe that $L$ is a regular language. Notice that, by the definition
of $R$, for each word $v$ in $K$ there is exactly one word $w \in L$
with $(v,w) \in R$. Since for each $x \in S$ there is exactly one word
$v$ in $K$ with $\elt{v} = x$, it follows that there is exactly one
word $w \in L$ with $\elt{w} = x$. That is, the language $L$ maps
bijectively onto $S$.
\end{proof}



\section{A non-Markov monoid with a regular set of unique representatives}
\label{sec:regnotmarkov}

This section exhibits a non-Markov monoid that nevertheless admits a
regular language of unique representatives over any alphabet
representing a finite generating set.  (That is, regularity and
uniqueness of representatives is achievable over any alphabet
representing a generating set, but prefix-closure is never achievable.)
This is important because it shows that the classes of Markov
semigroups and monoids are properly contained in the classes of
semigroups and monoids admitting regular languages of unique normal
forms: the requirement of prefix-closure properly restricts the
classes under consideration.

The example depends on the following construction from
\cite[\S~5]{maltcev_hopfian}.

\begin{definition}
For any action of a semigroup $S$ on a set $T$, define a new semigroup
$S[T]$ as follows. The carrier set is $S \cup T$; multiplication in
$S$ remains the same, and for $s \in S$ and $x,y \in T$,
\[
sx = x,\qquad xs = x\cdot s,\qquad xy = y.
\]
It is straightforward to check that this multiplication is
associative.
\end{definition}

To construct the example, proceed as follows. Let $F$ and $F'$ be free
monoids with bases $X = \{x,y\}$ and $X' = \{x',y'\}$ respectively and let
\[
R = \{w \in F' : \text{$|w|_{y'}$ is even}\}.
\]
Let $w_0,w_1,w_2,\ldots$ be the elements of $R$ enumerated in
length-plus-lexicographic order. Define $\psi : \nset \cup \{0\} \to
R$ by $j \mapsto w_j$, so that $\psi$ is a bijection between $\nset
\cup \{0\}$ and $R$. Notice that $|j\psi| < 2^j$ for all $j \in \nset
\cup \{0\}$. Let
\begin{align*}
P &= \{p_i : i \in \nset\}, \\
Q &= \{q_i : i \in \nset \land \neg(\exists j \in \nset \cup \{0\})(i = 2^j)\}, \\
T &= P \cup Q \cup F' \cup \{\Omega\}.
\end{align*}
Define an action of the generators $x$ and $y$ on the set $T$ as follows:
\begin{align*}
p_i \cdot x &= p_{i+1}, \\
p_i \cdot y &= \rlap{$\begin{cases} q_i & \text{if $i \neq 2^j$ for any $j \in \nset \cup \{0\}$}, \\
j\psi & \text{if $i = 2^j$},
\end{cases}$} \\
q_i \cdot x  &= \Omega, & w \cdot x &= wx' \text{~~(for $w \in F'$)}, & \Omega \cdot x  &= \Omega, \\
q_i \cdot y  &= \Omega, & w \cdot y &= wy' \text{~~(for $w \in F'$)}, & \Omega \cdot y  &= \Omega.
\end{align*}
\fullref{Figure}{fig:nonmarkovaction} illustrates the graph of the action of
$X$ on $T$. Since $F$ is free on $X$, this action extends to a
unique action of $F$ on $T$.

\begin{figure}[t]
\centerline{\includegraphics{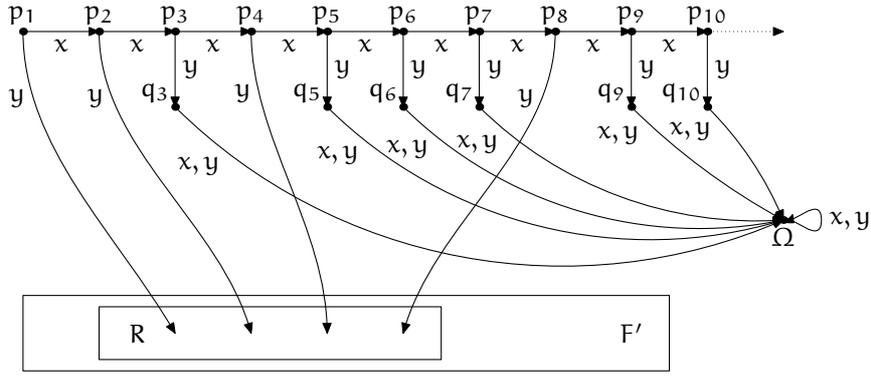}}
\caption{An outline of the graph of the action of $X$ on $T$.}
\label{fig:nonmarkovaction}
\end{figure}

The aim is to show that $F[T]$ is not Markov but nevertheless admits a
regular language of unique representatives over any finite alphabet
representing a generating set.

Notice that in $F[T]$, elements of $F$ multiply as in the free monoid
and act on $T$. Elements of $F'$ are members of the set $T$ and thus
multiply like right zeroes.

\begin{proposition}
The monoid $F[T]$ admits a regular language of unique representatives
over any finite alphabet representing a generating set.
\end{proposition}

\begin{proof}
By \fullref{Proposition}{prop:markovchangegen}, it suffices to prove
that $F[T]$ admits a regular language of unique representatives over
some particular finite alphabet representing a generating set.

Let $A = \{a,b,c,d,e,f\}$, where $\elt{a} = x$, $\elt{b} = y$, $\elt{c}
= x'$, $\elt{d} = y'$, $\elt{e} = p_1$, and $\elt{f} = \Omega$. Let $\rho : F' \to A^+$
be the bijection extending $x' \mapsto c$ and $y' \mapsto d$. Let
\[
L = \{a,b\}^* \cup ea^* \cup ea^*b \cup (\{c,d\}^+ - R\rho) \cup \{f\}.
\]
Then $L$ maps bijectively onto $F[T]$. In particular, the subset
$\{a,b\}^*$ maps bijectively onto $F$, the subset $ea^*$ maps
bijectively onto $\{p_i : i \in \nset\}$, the subset $ea^*b$ maps
bijectively onto $\{q_i : i \in \nset\} \cup R$, and the subset
$\{c,d\}^+ - R\rho$ maps bijectively onto $F' - R$. So $L \subseteq
A^*$ is a regular language of unique representatives for $F[T]$. [Note
  that $L$ is not prefix-closed, since it does not contain words from
  $R\rho$ but does contain all words in $(Ry')\rho = (R\rho)d$.]
\end{proof}

\begin{proposition}
The monoid $F[T]$ is not Markov.
\end{proposition}

\begin{proof}
Suppose, with the aim of obtaining a contradiction, that $F[T]$ admits
a Markov language $L$ over some alphabet $A$. 

Informally, the strategy is to reach a contradiction by proving the
following:
\begin{enumerate}

\item Sufficiently long elements of $R$ must have representatives in
  $L$ that label paths that run through $P$ for most of their length
  (excepting a short prefix) and enter $R \subseteq F'$ on their last
  letter. (\fullref{Lemma}{lem:rreps}.)

\item Sufficiently long elements of $F' - R$ have representatives in
  $L$ that label paths that run through $F'$ for most of their length
  (excepting a short prefix). (\fullref{Lemma}{lem:nonrreps}.)

\item Taking a suitable prefix of a representative of an element of
  $F'-R$ yields a representative of an element of $R$ that is not of
  the form described in step~1. (Conclusion of proof.)

\end{enumerate}

As as preliminary, define several subalphabets of $A$ and several
constants that will be used later to clarify what `sufficiently long'
means in the plan above. Let
\begin{align*}
A_P &= \{a \in A : \elt{a} \in P\}, \\
A_Q &= \{a \in A : \elt{a} \in Q\}, \\
A_{F'} &= \{a \in A : \elt{a} \in F'\}, \\
A_F &= \{a \in A : \elt{a} \in F\}, \\
A_x &= \{a \in A : \elt{a} \in x^+\}, \\
A_\Omega &= \{a \in A : \elt{a} = \Omega\}; 
\end{align*}
notice that $A$ is the disjoint union of $A_P$, $A_Q$, $A_{F'}$, $A_F$, and $A_\Omega$, and that $A_x \subseteq A_F$. Let
\begin{align*}
m_1 &= |u|, \text{where $u$ is the unique representative in $L$ of $\Omega$}, \\
m_2 &= \max\{i : p_i \in \elt{A_P}\}, \\
m_3 &= \max\{i : q_i \in \elt{A_Q}\}, \\
m_4 &= \max\{|\elt{a}| : a \in A_{F'}\}, \\
m &= \max\{m_1,m_2,m_3,m_4\}.
\end{align*}
Let $k = \max\{|\elt{a}| : a \in A_F\}$.

Let $\fsa{A}$ be a deterministic finite automaton recognizing
$L$. Consider the set of labels on simple loops in $\fsa{A}$. Let $V$
be the set of such labels that lie in $A_x^*$. Let $n$ be a constant
that is a multiple of all of the lengths of the elements of $V$ and
that also exceeds the number of states in $\fsa{A}$.

\begin{lemma}
\label{lem:shortprefix}
Let $uav \in L$, where $a \in A - A_F$. Then $|u| < n$. That is, any
letter from $A_P \cup A_Q \cup A_{F'} \cup A_\Omega$ in a word in $L$
must lie in the first $n$ letters, and hence $L \subseteq A^{\leq n}A_F^*$.
\end{lemma}

\begin{proof}
Suppose for \textit{reductio ad absurdum} that $uav \in L$ is as in
the hypothesis but that $|u| > n$. Then by the pumping lemma, $u$
factorizes as $u'u''u'''$ such that $u'(u'')^\alpha u'''av \in L$ for
all $\alpha \in \nset \cup \{0\}$. Since $\elt{a} \in T$, it follows
from the definition of multiplication in $F[T]$ that
\[
\elt{u'(u'')^\alpha u'''av} = \elt{av},
\]
for every $\alpha \in \nset \cup \{0\}$, which contradicts the
uniqueness of representatives in $L$. Hence $|u| \leq n$.
\end{proof}

\begin{lemma}
\label{lem:rreps}
The representative in $L$ of every $w \in R \subseteq F'$ with $|w| > m+n+k+kn$ has
the form $vc$, where $v \in A^*$, $c \in A_F-A_x$, $\elt{v} \in P$ and $\elt{c} = x^\beta y$ for
some $\beta < k$.
\end{lemma}

\begin{proof}
Let $j$ be such that $j\psi = w$. Since $|w| > m+n+k+kn$, it follows
that $2^j > |w| > m+n$ and hence $2^j - n > m$. It also follows that
$2^j > n+kn \geq 2n$, and so $n < 2^{j-1}$. Hence $2^j - n >
2^{j-1}$. Thus $2^j - n$ is not a power of $2$ and so there is an
element $q_{2^j-n} \in Q$.

Let $t$ be the representative in $L$ of $q_{2^j-n}$. Since $2^j - n >
m$, the rightmost letter $a$ from $A - A_F$ in the word $t$ cannot be
such that $\elt{a} = q_{2^j - n}$ by the definition of $m$; therefore
$a$ must lie in $A_P$. By \fullref{Lemma}{lem:shortprefix}, $t$
factorizes as $uas$, where $|u| < n$ and $s \in A_F^*$. Let $s =
s'cs''$, where $s' \in A_x^*$ and $c \in A_F - A_x$. (Such a letter $c$
must exist, otherwise $\elt{ubs} \in P$.) Now, $\elt{uas'c} \in
Q$. Since the action of $F$ on any element of $Q$ leads to the sink
element $\Omega$, it follows that $s''$ is the empty word. Hence $t =
uas'c$.

Let $\elt{c} = x^\beta yz$, where $z \in \{x,y\}^*$. Then
$\elt{uas'}x^\beta y \in Q$, and so $z = \emptyword$ since otherwise
$\elt{uas'}x^\beta yz = \Omega$. Since $|\elt{c}| \leq k$, it follows
\textit{a fortiori} that $\beta < k$.

Furthermore, since $\elt{uas'c} = q_{2^j-n}$, it follows that
$\elt{uas'} = p_{2^j - n -\beta}$. Hence, since $\elt{ua} = \elt{a} =
p_{m'}$ for some $m' \leq m$, it follows that
\[
|\elt{s'}| = 2^j-n-\beta-m' \geq 2^j -n -k-m > kn.
\]
Thus $|s'| > n$ since each letter of $s'$ represents an element of $F$
whose length is at most $k$.

Thus by the pumping lemma $s'$ factorizes as $v'v''v'''$, where
$|\elt{v''}|$ divides $n$ (by the definition of $n$) and
$v'(v'')^\alpha v''' \in L$ for all $\alpha \in \nset \cup \{0\}$. Set
$\alpha = n/|\elt{v''}|+1$. Then $\elt{uav'(v'')^\alpha v'''} =
p_{2^j-\beta}$. Thus
\[
\elt{uav'(v'')^\alpha v'''c} = p_{2^j-\beta}x^\beta y = p_{2^j}y = p_{2^j}\psi = w.
\]
Set $v = uav'(v'')^\alpha v'''$ to see that the representative $t$ of
$w$ has the form $vc$.
\end{proof}

\begin{lemma}
\label{lem:nonrreps}
Let $w \in F' - R$. Then the representative in $L$ of $w$ factorizes
as $uv$ where $\elt{u} \in F'$ with $|\elt{u}| < m+k+kn$ and $v \in
A_F^*$.
\end{lemma}

\begin{proof}
Let $w$ be in the hypothesis and let $t$ be its representative in
$L$. Since $t$ cannot lie in $A_F^*$, it contains some letter from
$A_P \cup A_Q \cup A_{F'} \cup A_\Omega$. The rightmost such letter
cannot lie in $A_Q \cup A_\Omega$, since this would force $\elt{t}$ to
lie in $Q \cup \{\Omega\}$. So the rightmost such letter is either
from $A_P$ or $A_{F'}$.

If the rightmost such letter is from $A_{F'}$, then by
\fullref{Lemma}{lem:shortprefix}, $t = u'av$, where $a \in A_{F'}$,
$|u'| < n$, $v \in A_F^*$. Set $u = u'a$. Then $\elt{u} = \elt{a}$ and
so $|\elt{u}| < m < m+nk$ and there is nothing more to prove.

So suppose the rightmost such letter is from $A_P$. Then by
\fullref{Lemma}{lem:shortprefix}, $t = t'bt''$, where $b \in A_P$,
$|t'| < n$, $t'' \in A_F^*$. Then $w = \elt{t} = \elt{bt''}$. Now, if
$t'' \in A_x^*$, then $\elt{bt''} \in P$ by the definition of the
action. So $t''$ contains some letter from $A_F - A_x$. Let $t'' =
scv$, where this distinguished letter $c$ is the leftmost letter of
$t''$ that is from $A_F - A_x$, so that $s \in A_x^*$. Then $\elt{bs} \in
P$ and $\elt{bsc} \in F'$ since the alternative $\elt{bsc} \in Q \cup
\{\Omega\}$ cannot happen since this set is closed under the action of
$v$.

Thus far $t$ has been factorized as $t'bscv$. The next step is to show
that $|s| < n$.  Suppose for \textit{reductio ad absurdum} that $|s|
\geq n$. Then $s$ factorizes as $s's''s'''$, where
$t'bs'(s'')^\alpha s'''cv \in L$ for all $\alpha \in \nset \cup
\{0\}$. Now, since $s = s's''s''' \in A_x^*$, the elements
$\elt{t'bs'(s'')^\alpha s'''}$ are a sequence of elements
$p_{i_\alpha}$ whose indices $i_\alpha$ form a linear progression. But
the indices of the elements $p_i \in P$ such that $p_i \cdot \elt{c}
\in F'$ are the terms of an exponential function. So there are infinitely many
$\alpha \in \nset \cup \{0\}$ such that $\elt{t'bs'(s'')^\alpha s'''c} = q_j \in Q \cup \{\Omega\}$.

Reasoning as in the third paragraph of the proof of
\fullref{Lemma}{lem:rreps}, $c = x^\beta y$. Now, if $v \neq
\emptyword$, then $\elt{t'bs'(s'')^\alpha s'''cv} = \Omega$ for infinitely many
$\alpha \in \nset \cup \{0\}$, which contradicts uniqueness of
representatives. If, on the other hand, $v = \emptyword$, then $w =
\elt{t} = \elt{t'bsc} = j\psi$ for some $j$ since $\elt{t'bs} \in P$,
$\elt{c} = x^\beta y$, and $\elt{t'bsc} \in F'$. So $w = j\psi \in R$,
which contradicts the hypothesis of the lemma. Hence $|s| < n$.

Therefore $|\elt{s}| < kn$ since each letter
of $s$ represents a word in $A_x^*$ of length at most $k$.

Now, $\elt{b} = p_{m'}$, where $m' < m$ by the definition of
$m$. Hence $\elt{bs} = p_{m'}\elt{s} = p_h$ for some $h < m+kn$ by the
definition of the action of $x$ on the $p_i$. Suppose $c = x^\beta y
z$ for some and $z \in \{x,y\}^*$. Then $\beta + |z| < k$. Since $\elt{t'bsc} \in
F'$, it follows that $h+\beta = 2^j$ for some $j \in \nset \cup \{0\}$. Hence $\elt{t'bsc}
= wz$, where $w \in R$ with $|w| < 2^j$. Now,
\[
|\elt{t'bsc}| = |wz| = |w| + |z| < |j\psi| + |z| < 2^j + |z| = h+\beta + |z| < h +k < m+k +kn.
\]
Let $u = t'bsc$. Then $t = uv$ with $|\elt{u}| < m+k+kn$.
\end{proof}

Choose $w \in R$ with $|w| > m+n+k+kn$. Then $|w|_{y'}$ is even and so
$|w(x')^{2k}y'|_{y'}$ is odd, so that $w(x')^{2k}y' \notin R$. Let $t$ be the
representative in $L$ of $w(x')^{2k}y'$. Then by
\fullref{Lemma}{lem:nonrreps}, $t$ factorizes as $uv$, where the $v$
is the longest suffix lying in $A_F^*$ and $\elt{u} \in F'$ with
$|\elt{u}| < m+k+kn$

In particular, $|w(x')^{2k}y'| > m+3k+kn$. Since $|\elt{u}| < m+k+kn$,
it follows that $|\elt{v}| > 2k$. Since each letter of $v$ represents
an element of $F$ of length at most $k$, the word $v$ has length at
least $2$. So let $v = v'ab$, where $a,b \in A_F$.  Since
$|\elt{a}|,|\elt{b}| < k$, $\elt{t} = \elt{uv} = w(x')^{2k}y'$ and
$\elt{u} \in F'$, it follows from the action of $F$ on $F' \subseteq
T$ that $\elt{uv'} = w(x')^\alpha$ for some $\alpha \in
\{1,\ldots,2k\}$, $\elt{a} = x^\beta$ (so that $a \in A_x$), and $b
\in A_F - A_x$.

Let $t' = uv'a$. Then $\elt{t'} = w(x')^{\alpha+\beta}$. By
prefix-closure, $t' \in L$. Observe that $t'$ ends with $a \in A_x$.

Now, the word $w(x')^{\alpha+\beta}$ lies in $R$ since $|w|_{y'} =
|w(x')^{\alpha+\beta}|_{y'}$ is even. So by \fullref{Lemma}{lem:rreps}, its
unique representative $t'$ must factorize as $sc$, where $\elt{c} =
x^\beta y$, so that $c \in A_F - A_x$. This contradicts the fact that
$\elt{t'}$ ends with a letter from $A_x$.

Thus $F[T]$ does not admit a Markov language.
\end{proof}

\section{Rewriting systems}
\label{sec:rs}

Confluent noetherian rewriting systems form a natural source of
examples of Markov semigroups. The following result is easily noticed,
but will prove very useful:

\begin{proposition}
\label{prop:rsmarkov}
Let $(A,\rel{R})$ be a confluent noetherian rewriting system with the
set of left-hand sides of rewriting rules in $\rel{R}$ being
regular. Then the monoid presented by $\pres{A}{\rel{R}}$ is Markov,
and its language of normal forms is a Markov language. Furthermore, if
$(A,\rel{R})$ is non-length-increasing, then the language of normal
forms is a robust Markov language for the monoid.
\end{proposition}

\begin{proof}
The language $L = A^* - \{\ell : (\ell,r) \in \rel{R}\}$, which is the
language of normal forms of $(A,\rel{R})$, is regular, prefix-closed,
and maps bijectively onto the monoid presented by
$\pres{A}{\rel{R}}$. For the final observation, notice that if
$(A,\rel{R})$ is non-length-increasing, then the language of normal
forms consists of minimal-length representatives.
\end{proof}

It is worth emphasizing that \fullref{Proposition}{prop:rsmarkov} says
that being Markov is a necessary condition for a semigroup to be
presented by a confluent noetherian rewriting system, although it is
probably not as useful as other necessary conditions such as finite
derivation type \cite{squier_fdt}, which are independent of the choice
of generating set.

However, the following example shows that a semigroup presented by a
finite confluent noetherian non-length-increasing rewriting system can
admit a robust Markov language that looks very different from its
language of normal forms:

\begin{example}
Let $A = \{a,b\}$ and $\rel{R} = \{(a^2,ba),(b^2,ab)\}$. Then
$(A,\rel{R})$ is confluent and noetherian. Let $L$ be its language of
normal forms; this is a robust Markov language by
\fullref{Proposition}{prop:rsmarkov}. Then $L$ is the language of
words over $A$ that do contain neither two consecutive letters $a$ nor
two consecutive letters $b$; thus $L$ is the language of alternating
products of letters $a$ and $b$:
\begin{align*}
L &= (A^* - A^*aaA^*) - A^*bbA^* \\
&= (ab)^* \cup (ab)^*a \cup (ba)^* \cup (ba)^*b.
\end{align*}
Let $M$ be the monoid presented by $\pres{A}{\rel{R}}$. Let
\[
K = ab^* \cup ba^*.
\]
The aim is to show that $K$ is also a Markov language for $M$. Notice
first that $K$ is prefix-closed and regular and so it remains to show
that it consists of unique minimal-length representatives for $M$.

Notice that for any $\alpha \in \nset \cup \{0\}$,
\[
\elt{(ab)^\alpha} = \elt{ab(ab)^{\alpha-1}} = \elt{ab(b^2)^{\alpha - 1}} = \elt{ab^{2\alpha - 1}}
\]
and
\[
\elt{(ab)^\alpha a} = \elt{ab(ab)^{\alpha-1}a} = \elt{bb(ab)^{\alpha - 1}a} = \elt{b(ba)^\alpha} = \elt{b(a^2)^\alpha} = \elt{ba^{2\alpha}}.
\]
Parallel reasoning shows that $\elt{(ba)^\alpha} =
\elt{ba^{2\alpha-1}}$ and $\elt{(ba)^\alpha b} =
\elt{ab^{2\alpha}}$. Thus every word in $L$ represents the same
element as exactly one element of $K$ and vice versa. Furthermore, the
lengths of the corresponding words in $L$ and $K$ are the same. Hence,
since $L$ is a robust Markov language for $M$ by
\fullref{Proposition}{prop:rsmarkov}, $K$ is also a robust Markov
language for $M$.
\end{example}

\begin{question}
Is every Markov semigroup presented by a confluent noetherian
rewriting system where the language of left-hand sides of rewriting
rules is regular? (That is, where the language of \emph{all} left-hand sides
is regular: \fullref{Example}{ex:wpunsolvable} below shows that the
language of left-hand sides of rules with a particular right-hand side
may be irregular.)
\end{question}

\section{Markov, robustly Markov, and strongly Markov semigroups}

The example in \fullref{\S}{sec:regnotmarkov} consists of a
non-Markov monoid that admitted a regular language of unique
representatives over any alphabet representing a generating set. The
present section gives an example of a monoid that is Markov but not
robustly Markov (\fullref{Example}{ex:markovnotrobustlymarkov}) and an
example of a monoid that is robustly Markov but not strongly Markov
(\fullref{Example}{ex:robustlymarkovnotstronglymarkov}). These three
examples together show that the classes of Markov, robustly Markov,
and strongly Markov semigroups are distinct.

\begin{example}
\label{ex:markovnotrobustlymarkov}
Let
\begin{align*}
P &= \{p_i : i \in \nset\}, \\
Q &= \{q_i : i \in \nset \land \neg(\exists j \in \nset)(i = 2^j)\}, \\
R &= \{r_i : i \in \nset\}, \\
S &= \{s_i : i \in \nset\}, \\
T &= P \cup Q \cup R \cup S \cup \{\Omega\}. 
\end{align*}
Let $F$ be a free monoid with basis
$X = \{x,y\}$. Define an action of $X$ on $T$ as follows
\begin{align*}
p_i \cdot x &= p_{i+1}, &
p_i \cdot y &= \begin{cases}
q_i & \text{if $i \neq 2^j$ for any $j \in \nset \cup \{0\}$}, \\
s_j & \text{if $i = 2^j$ for some $j \in \nset \cup \{0\}$}, 
\end{cases} \\
q_i \cdot x &= \Omega, &
q_i \cdot y &= \Omega, \\
r_i \cdot x &= r_{i+1}, &
r_i \cdot y &= s_i, \\
s_i \cdot x &= \Omega, &
s_i \cdot y &= \Omega, \\
\Omega \cdot x &= \Omega, &
\Omega \cdot y &= \Omega. 
\end{align*}
Since $F$ is free on $X$, this action extends to a unique action of
$F$ on $T$. \fullref{Figure}{fig:nonrobustmarkovaction} shows the
graph of the action of $X$ on
$T$. \fullref{Propositions}{prop:exmarkov} and
\ref{prop:exnotrobustlymarkov} below show that $F[T]$ is strongly Markov but
not robustly Markov.
\end{example}

\begin{figure}[t]
\centerline{\includegraphics{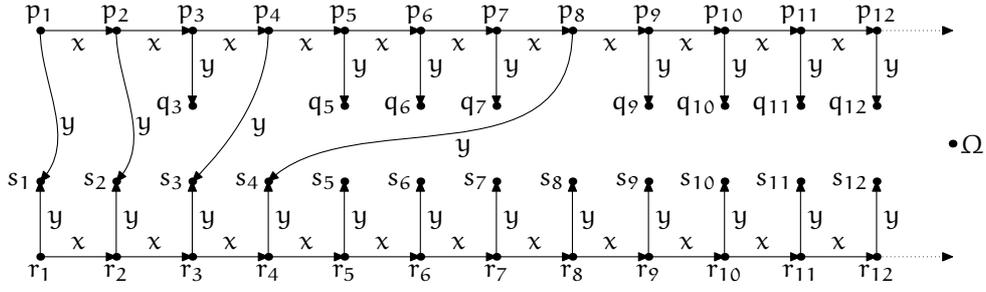}}
\caption{Part of the graph of the action of $X$ on $T$. Edges
  which lead to $\Omega$ are not shown}
\label{fig:nonrobustmarkovaction}
\end{figure}

\begin{proposition}
\label{prop:exmarkov}
The monoid $F[T]$ is Markov.
\end{proposition}

\begin{proof}
Let $A = \{a,b,c,d,e\}$ be an alphabet representing elements of $F[T]$ as follows:
\[
\elt{a} = x,\quad \elt{b} = y,\quad \elt{c} = p_1,\quad\elt{d} = r_1,\quad\elt{e} = \Omega.
\]
Let $K = \{a,b\}^* \cup ca^* \cup ca^*b \cup da^* \cup \{e\}$. Then $K$ is
prefix-closed, regular, and maps bijectively onto $F[T]$. In
particular, the subset $\{a,b\}^*$ maps bijectively onto $F$, the
subset $ca^*$ maps bijectively onto $P$, the subset $ca^*b$ maps
bijectively onto $Q \cup S$, and the subset $da^*$ maps bijectively
onto $R$. Thus $K$ is a Markov language for $F[T]$.
\end{proof}

\begin{proposition}
\label{prop:exnotrobustlymarkov}
The monoid $F[T]$ is not robustly Markov.
\end{proposition}

\begin{proof}
Suppose, with the aim of obtaining a contradiction, that $F[T]$ admits
a robust Markov language $L$ over some alphabet $A$. 

Define the following subalphabets of $A$:
\begin{align*}
A_P &= \{a \in A : \elt{a} \in P\}, \\
A_Q &= \{a \in A : \elt{a} \in Q\}, \\
A_R &= \{a \in A : \elt{a} \in R\}, \\
A_S &= \{a \in A : \elt{a} \in S\}, \\
A_F &= \{a \in A : \elt{a} \in F\}, \\
A_x &= \{a \in A : \elt{a} \in x^+\}, \\
A_\Omega &= \{a \in A : \elt{a} = \Omega\}; 
\end{align*}
notice that $A$ is the disjoint union of $A_P$, $A_Q$, $A_R$, $A_S$, $A_F$, and $A_\Omega$. Let
\begin{align*}
m_1 &= \max\{i : p_i \in \elt{A_P}\}, \\
m_2 &= \max\{i : q_i \in \elt{A_Q}\}, \\
m_3 &= \max\{i : r_i \in \elt{A_R}\}, \\
m_4 &= \max\{i : s_i \in \elt{A_S}\}, \\
m &= \max\{m_1,m_2,m_3,m_4\}.
\end{align*}
Let $k = \max\{|\elt{a}| : a \in A_F\}$.

Reasoning as in the proof of \fullref{Lemma}{lem:rreps}, one sees that
for $i$ sufficiently large, $s_i$ is represented by a word of the form
$vc$, where $v \in A^*$, $c \in A_F - A_x$, $\elt{v} \in P$, and
$\elt{c} = x^\beta y$ for some $\beta < k$.

Let $v = v'bv''$, where $v'' \in A_F$. Then $b \in A_P$ and so
$\elt{v'b} = \elt{b} = p_{m'}$ for some $m' < m$. Now, $s_i = \elt{vc}
= \elt{v'bv''c} = p_{m'}\elt{v''}x^\beta y$, and so by the definition
of the action, $p_{2^i} = p_{m'}\elt{v''}x^\beta$. Thus $\elt{v''} =
s^{2^i - m' - \beta}$. So each letter of $v''$ lies in
$A_x$. Furthermore, since each such letter represents an element of
length at most $k$, it follows that $|v''| > (2^i - m -
\beta)/k$ and further that $|v| > (2^i - m - \beta)/k + 2$.

Since $v'' \in A_x^*$, the subalphabet $A_x$ must be non-empty. Let $a
\in A_x$, with $\elt{a} = x^\gamma$. Since $(T-Q) \cdot F$ does not
contain any element of $Q$, the subalphabet $A_Q$ is non-empty and
contains some letter $b$ with $\elt{b} = q_\alpha$.

Then $\elt{ba^hc} = q_\alpha x^{\gamma h + \beta}y = q_{\alpha+\gamma
  h+\beta}y = s_{\alpha+\gamma h + \beta}$. By choosing $h$ large
enough, $s_{\alpha+\gamma h + \beta}$ is represented in $L$ by a word $v$ of length greater
than $(2^{\alpha+\gamma h + \beta} - m -\beta)/k + 2$. Again choosing $h$ large enough, so that
\[
(2^{\alpha+\gamma h +\beta} - m -\beta)/k + 2 > h + 2.
\]
one obtains $|v| > |ba^hc|$. Thus $v$ is not a minimal-length representative
of $s_{\alpha+\gamma h + \beta}$, which contradicts $L$ being a robust
Markov language for $F[T]$.
\end{proof}

\begin{example}
\label{ex:robustlymarkovnotstronglymarkov}
Let
\begin{align*}
P &= \{p_i : i \in \nset \cup \{0\}\}, \\
Q &= \{q_i : i \in \nset\}, \\
R &= \{r_i : i \in \nset\}, \\
T &= P \cup Q \cup R. 
\end{align*}
Let $F$ be a free monoid with basis
$X = \{x,y,z\}$. Define an action of $X$ on $T$ as follows
\begin{align*}
p_i \cdot x &= p_{i+1}, &
q_i \cdot x &= \begin{cases}
q_i & \text{if $i \neq 2^j$ for any $j \in \nset \cup \{0\}$}, \\
r_i & \text{if $i = 2^j$ for some $j \in \nset \cup \{0\}$}, 
\end{cases} &
r_i \cdot x &= r_i, \\
p_i \cdot y &= q_i, &
q_i \cdot y &= q_i, &
r_i \cdot y &= r_i, \\
p_i \cdot z &= p_i, & 
q_i \cdot z &= \begin{cases}
q_i & \text{if $i = 2^j$ for some $j \in \nset \cup \{0\}$}, \\
r_i & \text{if $i \neq 2^j$ for any $j \in \nset \cup \{0\}$}, 
\end{cases} &
r_i \cdot z &= r_i.
\end{align*}
(Notice that $q_i$ is fixed by one of $x$ or $z$ and sent to $r_i$ by
the other, and that which letter fixes $q_i$ and which sends it to
$r_i$ depends on whether $i$ is a power of $2$.) Since $F$ is free on
$X$, this action extends to a unique action of $F$ on
$T$. \fullref{Figure}{fig:nonstrongmarkovaction} shows the graph of
the action of $X$ on
$T$. \fullref{Propositions}{prop:exrobustlymarkov} and
\ref{prop:exnotstronglymarkov} below show that $F[T]$ is robustly Markov but
not strongly Markov.
\end{example}

\begin{figure}[t]
\centerline{\includegraphics{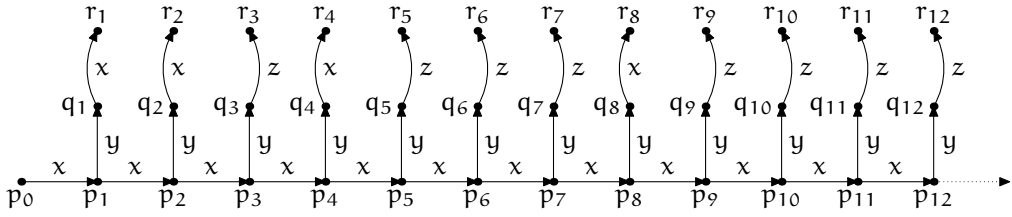}}
\caption{Part of the graph of the action of $X$ on $T$. Edges
  corresponding to actions which fix elements of $T$ are not shown}
\label{fig:nonstrongmarkovaction}
\end{figure}

\begin{proposition}
\label{prop:exrobustlymarkov}
The monoid $F[T]$ is robustly Markov.
\end{proposition}

\begin{proof}
Let $A = \{a,b,c,d,e,f\}$ be an alphabet representing elements of $F[T]$ as follows:
\[
\elt{a} = x,\quad\elt{b} = y,\quad\elt{c} = z,\quad\elt{d} = yx,\quad\elt{e} = yz,\quad\elt{f} =p_0.
\]
Let $A' = A - \{f\}$. Then $(A',\{(ba,d),(bc,e)\})$ is a confluent
noetherian rewriting system presenting the subsemigroup $F$ of
$F[T]$. Hence its language of normal forms $K_1 = A^* - A^*(ba \cup
bc)A^*$ is a robust Markov language for the subsemigroup $F$ of $F[T]$
by \fullref{Proposition}{prop:rsmarkov}.

Let $K_2 = fa^* \cup fa^+d \cup fa^+e$. Then $K_2$ is
$+$-prefix-closed and regular. The subset $fa^*$ maps bijectively onto
$P$. The subsets $fa^+d$ and $fa^+e$ map bijectively onto $Q \cup R$,
since for each $i \in \nset$, exactly one of the following cases
holds:
\begin{itemize}

\item $\elt{fa^id} = p_0x^iyx = p_iyx = q_ix = r_i$ and $\elt{fa^ie} = p_0x^iyz = p_iyz = q_iz = q_i$ (this holds if $i = 2^j$ for some $j \in \nset$);
\item $\elt{fa^id} = p_0x^iyx = p_iyx = q_ix = q_i$ and $\elt{fa^ie} = p_0x^iyz = p_iyz = q_iz = r_i$ (this holds if $i \neq 2^j$ for any $j \in \nset$).

\end{itemize}
Thus $K_2$ maps bijectively onto $T$.

It remains to show that every word in $K_2$ is a minimal length
representative. Let $u \in A^*$ represent $p_i$. Then $u$ must contain
$f$, since all other letters in $A$ represent elements of $F$. So let
$u = u'fu''$, where $u'' \in (A - \{f\})^*$, so that this
distinguished letter $f$ is the rightmost such letter in $u$. Each
symbol in $A -\{f\}$ represents an element of $F$ that contains at
most one letter $x$. So, by the definition of the action on the $p_i$,
it follows that $u''$ must contain at least $i$ letters. Hence $|u|
\geq i + 1$. Any word over $A$ representing $q_i$ or $r_i$ must
therefore have length at least $i+2$. By the observations in the
preceding paragraph, the representative in $K_2$ of $p_i$ has length
$i+1$, and those of $q_i$ and $r_i$ both have length $i+2$.

Therefore the language $K_1 \cup K_2$ is prefix-closed, regular, and
consists of minimal-length representatives for $F[T]$. So $K_1 \cup
K_2$ is a robust Markov language for $F[T]$.
\end{proof}

\begin{proposition}
\label{prop:exnotstronglymarkov}
The monoid $F[T]$ is not strongly Markov.
\end{proposition}

\begin{proof}
Suppose, with the aim of obtaining a contradiction, that $F[T]$ is
strongly Markov. Let $A = \{a,b,c,f\}$ represent elements of $F[T]$ as
follows:
\[
\elt{a} = x,\quad\elt{b} = y,\quad\elt{c} = z,\quad\elt{f}=p_0.
\]
Since $F[T]$ is strongly Markov, it admits a robust Markov language
$L$ over the alphabet $A$. Let $n$ be greater than the number of
states in an automaton recognizing $L$. Choose $k$ such that $2^k > n$.

It is easy to see that the unique shortest word over $A$ representing
$r_{2^k}$ is $fa^{2^k}ba$. Therefore this word lies in $L$. By the
pumping lemma, $a^{2^k}$ factorizes as $v'v''v'''$, where $v', v'',
v''' \in a^*$ and $fv'(v'')^\alpha v'''ba \in L$ for every $\alpha \in
\nset \cup \{0\}$. Choose $\alpha$ so that $m = |v'(v'')^\alpha v'''|$
is not a power of $2$. Then $\elt{fv'(v'')^\alpha v'''b} = q_m$, and
$\elt{fv'(v'')^\alpha v'''ba} = q_mx = q_m$. Hence $fv'(v'')^\alpha
v'''b$ and $fv'(v'')^\alpha v'''ba$ represent the same element of
$F[T]$. Since both these words lie in $L$ by prefix-closure, this
contradicts the uniqueness of representatives in $L$.
\end{proof}

\section{Commutative semigroups}
\label{sec:comm}

That finitely generated commutative semigroups are Markov could be
deduced from \fullref{Proposition}{prop:rsmarkov}, and the fact that
finitely generated commutative monoids have presentations via finite
confluent noetherian rewriting systems \cite{diekert_commutative}, and
the closure of the class of Markov semigroups under adjoining and
removing an identity (\fullref{Proposition}{prop:adjone} below). However, a
stronger result holds:

\begin{proposition}
\label{prop:commmarkov}
Finitely generated commutative semigroups are strongly Markov.
\end{proposition}

[The first part of the following proof parallels the proof that all
  commutative cancellative semigroups are automatic; see
  \cite[Theorem~5.4.2]{c_phdthesis}.]

\begin{proof}
Let $A$ be a finite alphabet representing an arbitrary generating set
for some commutative semigroup $S$.  Suppose $A =
\{a_1,\ldots,a_n\}$. Consider elements of $S$ using tuples: identify
the tuple $(\alpha_1,\ldots,\alpha_n)$ with the element
$\elt{a_1^{\alpha_1}\cdots a_n^{\alpha_n}}$. Define the ShortLex
ordering $\slex\prec$ of these tuples by
\begin{align*}
(\alpha_1,\ldots,\alpha_n) \slex\prec (\beta_1,\ldots,\beta_n) \iff& \sum_{i=1}^n \alpha_i < \sum_{i=1}^n \beta_i,\text{ or}\\
&\bigg[\sum_{i=1}^n \alpha_i = \sum_{i=1}^n \beta_i\\
&\qquad\text{ and }(\alpha_1, \ldots, \alpha_n) \lex\sqsubset (\beta_1, \ldots, \beta_n)\bigg],
\end{align*}
where $\lex\sqsubset$ is the lexicographical order of tuples:
$(\alpha_1, \ldots, \alpha_n) \lex\sqsubset (\beta_1, \ldots,
\beta_n)$ if the leftmost non-zero co\"{o}rdinate of $(\beta_1 -
\alpha_1, \ldots, \beta_n - \alpha_n)$ is positive.
 
R{\'{e}}dei's Theorem~\cite{redei_commutative_german} asserts that $S$
is finitely presented. An approach to this theorem found in
\cite[Chapter~5]{rosales_commutative} (which is a modification of
the proof in \cite{grillet_redei}) shows that the semigroup $S$
is isomorphic to
\[
\left[(\nset \cup \{0\})^n - \{(0,\ldots,0)\}\right]/\cgen{\{(u_1,v_1),\ldots,(u_n,v_n)\}},
\]
where $u_i \slex\prec v_i$, and such that the ShortLex-minimal
representative of $\elt{w} \in (\nset \cup \{0\})^n -
\{(0,\ldots,0)\}$ can be found by repeatedly replacing $w$ by $w - v_i
+ u_i$ whenever every co{\"{o}}rdinate of $w - v_i$ is
non-negative. (Addition is performed componentwise on tuples.)

Since the ShortLex order is compatible with the operation (that is,
for all $x \in S$, $u \slex\prec v \implies u+x \slex\prec v+x$), the
set of ShortLex-minimal elements is simply
\[
M = \left\{w \in (\nset \cup \{0\})^n - \{(0,\ldots,0)\} : \text{$w - v_i$ is not in $(\nset \cup \{0\})^n$ for any $i$}\right\}.
\]
Let
\[
K = \{a_1^{\alpha_1}\cdots a_n^{\alpha_n} : (\alpha_1, \ldots, \alpha_n) \in M\}.
\]
Since the number of $v_i$ is finite, a finite state automaton can
check whether a word $a_1^{\alpha_1}\cdots a_n^{\alpha_n}$ lies in
$K$. Therefore $K$ is regular.

Finally, notice that if a word $a_1^{\alpha_1}\cdots a_n^{\alpha_n}$
lies in $K$, then one obtains its longest proper prefix by decreasing
by $1$ the right-most non-zero exponent $\alpha_i$. (Recall that some
of the $\alpha_i$, but not all, can be $0$.) Thus if $w$ is the tuple
in $M$ corresponding to a word in $K$, then the tuple $w'$
corresponding to its longest proper prefix is obtained by decreasing
the right-most non-zero co\"ordinate by $1$. Hence if $w - v_i \notin
(\nset \cup \{0\})^n$ then $w' - v_i \notin (\nset \cup
\{0\})^n$. Consequently $K$ is closed under taking longest proper
non-empty prefixes, and so, by iteration, is $+$-prefix-closed. By the
definition of the ShortLex ordering, the language $K$ consists of
minimal-length representatives. So $K$ is a robust Markov language for
$S$. Since the generating set represented by $A$ was arbitrary, $S$ is
strongly Markov.
\end{proof}

Finitely generated abelian groups are Markov, as a consequence of the
more general result that finitely generated polycyclic groups are
Markov \cite[Corollaire~11]{ghys_markov}. However, that finitely
generated abelian groups are strongly Markov (an immediate corollary
of \fullref{Proposition}{prop:commmarkov}) does not seem to have been
explicitly noted anywhere, although it is implicit in
\cite[Chs~3--4]{epstein_wordproc}.

\section{Virtually abelian, nilpotent, and polycyclic groups}
\label{sec:polycyclic}

It is known that nilpotent groups need not be strongly Markov, since
they may have irrational (indeed, transcendental) growth functions
with respect to some generating
sets~\cite[Theorem~B]{stoll_growthseries}. Furthermore, there exist
virtually abelian groups that do not admit any regular language of
minimal length representatives over some generating set (that is, even
without requiring uniqueness) \cite{neumann_regulargeodesic}. Thus
virtually abelian groups are not in general strongly Markov.

\begin{question}
Are finitely generated semigroups that are nilpotent (in the sense of
Malcev \cite{malcev_nilpotent}) Markov? In particular, are all
finitely generated subsemigroups of nilpotent groups are Markov?
\end{question}

This section exhibits two examples to show that finitely generated
subsemigroups of virtually abelian groups and of polycyclic groups
need not be Markov. All finitely generated sub\emph{groups} of such
groups are Markov, since these classes of groups are closed under
taking subgroups.

The example of a non-Markov subsemigroup of a virtually abelian group
(\fullref{Example}{ex:vabel}) is particularly important: First, it
shows that the class of groups all of whose finitely generated
subsemigroups are Markov is not closed under forming finite
extensions. Second, virtually abelian groups satisfy a non-trivial
semigroup identity and thus have the following property: if $S$ is a
subsemigroup and $H$ the subgroup it generates, then $H$ is
[isomorphic to] the universal group of $S$. In general groups, this is
not true: $H$ is in general a homomorphic image of the universal group
of $S$. (The universal group of $S$ is the group obtained by taking a
presentation for $S$ and considering it as a group presentation; see
\cite[Ch.~12]{clifford_semigroups2} or the discussion in
\cite[\S~5.2.1]{c_phdthesis} for background information.) Thus the
example is a non-Markov semigroup with a Markov universal group.

The following technical result will be used in proving both examples non-Markov:

\begin{lemma}
\label{lem:abcdefghijnonmarkov}
Let $S$ be a semigroup and $A = \{a,b,c,d,e,f,g,h,i,j\}$ an alphabet
representing a finite generating set for $S$. Suppose that for
$\alpha,\beta \in \nset \cup \{0\}$ with $\alpha \neq \beta$ the
following conditions hold:
\begin{enumerate}
\item The element represented by $ab^\alpha cd^\beta e$ is represented
  by no other word over $A$.
\item The element represented by $fg^\alpha hi^\beta j$ is represented
  by no other word over $A$.
\item The equality $\elt{ab^\alpha cd^\alpha e} = \elt{fg^\alpha
  hi^\alpha j}$ holds, and the only words representing this element
  are $ab^\alpha cd^\alpha e$ and $fg^\alpha hi^\alpha j$.
\end{enumerate}
Then $S$ is not Markov.
\end{lemma}

\begin{proof}
Suppose for \textit{reductio ad absurdum} that $S$ is Markov. Then it
admits a regular language of unique representatives $L$ over $A$ by
\fullref{Proposition}{prop:markovchangegen}. So $K=L\cap (ab^*cd^*e
\cup fg^*hi^*j)$ is regular. By assumption, when
$\alpha \neq \beta$, the element represented by $ab^\alpha cd^\beta e$
is represented by no other word over $A$. Thus $K \supseteq
\{ab^\alpha cd^\beta e:\alpha \neq \beta\}$ and similarly
$K\supseteq\{fg^\alpha hi^\beta j:\alpha\neq\beta\}$.  Thus $ab^*cd^*e
- K \subseteq \{ab^\alpha cd^\alpha e : \alpha \in \nset \cup \{0\}\}$
and $fg^*hi^*j - K \subseteq \{fg^\alpha hi^\alpha e : \alpha \in \nset
\cup \{0\}\}$.

Furthermore, the only representatives over $A$ of the element
$\elt{ab^\alpha cd^\alpha e}$ are $ab^\alpha cd^\alpha e$ and
$fg^\alpha hi^\alpha j$. So at least one of the regular languages
$ab^*cd^*e - K$ and $fg^*hi^*j - K$ is infinite. Assume the former; the
latter case is similar. Since $ab^*cd^*e - K \subseteq \{ab^\alpha
cd^\alpha e : \alpha \in \nset \cup \{0\}\}$ is infinite, it contains
arbitrarily long words $ab^\alpha cd^\alpha e$. So a string of symbols
$b$ can be pumped, which contradicts the fact that every word in this
language is of the form $ab^\alpha cd^\alpha e$. Thus $S$ is not
Markov.
\end{proof}

\begin{example}
\label{ex:polycyclic}
The semigroup presented by
\[
\pres{a,b,c,d,e,f,g,h,i,j}{ab^\alpha cd^\alpha e=fg^\alpha hi^\alpha j,
\alpha \in \nset \cup 0},
\]
which is isomorphic to a subsemigroup of a polycyclic group
\cite[\S~3]{c_ssgofdp}, is not Markov by
\fullref{Lemma}{lem:abcdefghijnonmarkov} above.
\end{example}

\begin{example}
\label{ex:vabel}
Let $\symgroup{11}$ be the symmetric group on eleven elements. Let
$\zset^{11}$ be the direct product of eleven copies of the integers under
addition. View elements of $\zset^{11}$ as $11$-tuples of integers. Let $G =
\symgroup{11} \ltimes \zset^{11}$, where $\symgroup{11}$ acts (on the
right) by permuting the components of elements of $\zset^{11}$. (The
$\zset$-components are indexed from $1$ at the left to $11$ at the
right.) The abelian subgroup $\zset^{11}$ of $G$ has index $11!$, so $G$
is a virtually abelian group.

Let $A = \{a,b,c,d,e,f,g,h,i,j\}$ be an alphabet representing elements
of $G$ in the following way:
\begin{align*}
\elt{a} &= [(1\;3),\;(0,1,1,0,0,0,1,0,0,0,0)],& \elt{f} &= [(1\;5),\;(0,1,0,0,1,0,1,0,0,0,0)],\\
\elt{b} &= [\id,\;(0,0,1,0,0,0,0,0,0,1,0)],& \elt{g} &= [\id,\;(0,0,0,0,1,0,0,0,0,0,1)],\\
\elt{c} &= [(1\;3)(2\;4),\;(1,0,0,0,0,0,0,1,0,0,0)],& \elt{h} &= [(1\;5)(2\;6),\;(1,0,0,0,0,0,0,1,0,0,0)],\\
\elt{d} &= [\id,\;(0,0,0,1,0,0,0,0,0,-1,0)],& \elt{i} &= [\id,\;(0,0,0,0,0,1,0,0,0,0,-1)],\\
\elt{e} &= [(2\;4),\;(0,1,0,0,0,0,0,0,1,0,0)],& \elt{j} &= [(2\;6),\;(0,1,0,0,0,0,0,0,1,0,0)].
\end{align*}
Let $S$ be the subsemigroup of $G$ generated by $\elt{A}$. The aim is
show that $S$ is not Markov. [We admit that the generators in
  $\elt{A}$ may look intimidating. However, they interact in a fairly
  nice way, and the method in their madness will become apparent.]

First of all, some preliminaries are necessary. For any $\alpha,\beta \in \nset \cup \{0\}$,
\begin{align*}
&\;\phantom{=}\;\elt{ab^\alpha cd^\beta e} \\
&= [(1\;3),\;(0,1,1,0,0,0,1,0,0,0,0)][\id,\;(0,0,\alpha,0,0,0,0,0,0,\alpha,0)]\elt{cd^\beta e} \\
&= [(1\;3),\;(0,1,\alpha+1,0,0,0,1,0,0,\alpha,0)][(1\;3)(2\;4),\;(1,0,0,0,0,0,0,1,0,0,0)]\elt{d^\beta e} \\
&= [(2\;4),\;(\alpha+2,0,0,1,0,0,1,1,0,\alpha,0)][\id,\;(0,0,0,\beta,0,0,0,0,0,-\beta,0)]\elt{e} \\
&= [(2\;4),\;(\alpha+2,0,0,\beta+1,0,0,1,1,0,\alpha-\beta,0)][(2\;4),\;(0,1,0,0,0,0,0,0,1,0,0)] \\
&= [\id,\;(\alpha+2,\beta+2,0,0,0,0,1,1,1,\alpha-\beta,0)],
\end{align*}
and
\begin{align*}
&\;\phantom{=}\;\elt{fg^\alpha hi^\beta j} \\
&= [(1\;5),\;(0,1,0,0,1,0,1,0,0,0,0)][\id,\;(0,0,0,0,\alpha,0,0,0,0,0,\alpha)]\elt{hi^\beta j} \\
&= [(1\;5),\;(0,1,0,0,\alpha+1,0,1,0,0,0,\alpha)][(1\;5)(2\;6),\;(1,0,0,0,0,0,0,1,0,0,0)]\elt{i^\beta j} \\
&= [(2\;6),\;(\alpha+2,0,0,0,0,1,1,1,0,0,\alpha)][\id,\;(0,0,0,0,0,\beta,0,0,0,0,-\beta)]\elt{j} \\
&= [(2\;6),\;(\alpha+2,0,0,0,0,\beta+1,1,1,0,0,\alpha-\beta)][(2\;6),\;(0,1,0,0,0,0,0,0,1,0,0)]\\
&= [\id,\;(\alpha+2,\beta+2,0,0,0,0,1,1,1,0,\alpha-\beta)].
\end{align*}
In particular, $\elt{ab^\alpha cd^\alpha e} = \elt{fg^\alpha hi^\alpha j}$.

\begin{lemma}
\label{lem:vabelreps}
Let $\alpha,\beta \in \nset \cup \{0\}$ with $\alpha \neq \beta$.
\begin{enumerate}
\item The only word over $A$ representing $\elt{ab^\alpha cd^\beta e}$
  is $ab^\alpha cd^\beta e$, and the only word over $A$ representing
  $\elt{fg^\alpha hi^\beta j}$ is $fg^\alpha hi^\beta j$.
\item The only words over $A$ representing $\elt{ab^\alpha cd^\alpha e} = \elt{fg^\alpha hi^\alpha j}$ are $ab^\alpha cd^\alpha e$ and $fg^\alpha hi^\alpha j$.
\end{enumerate}
\end{lemma}

\begin{proof}
Let $\alpha,\beta \in \nset \cup \{0\}$. For the present, allow the
possibility that $\alpha$ and $\beta$ are equal.

Let $w \in A^+$ be some word representing
\begin{equation}
\label{eq:virtabel1}
s = \elt{ab^\alpha cd^\beta e} = [\id,\;(\alpha+2,\beta+2,0,0,0,0,1,1,1,\alpha-\beta,0)].
\end{equation}
Let $A' = \{a,c,e,f,h,j\}$; observe that $A'$ consists of exactly
those elements of $A$ representing elements with non-zero seventh,
eighth, and ninth $\zset$-components, which are also exactly those
that have non-identity $\symgroup{11}$-components. Let $A'' = A - A' =
\{b,d,g,i\}$ observe that $A''$ consists of exactly those elements of
$A$ representing elements with non-zero tenth and eleventh
$\zset$-components, which are also exactly those that have identity
$\symgroup{11}$-components.

First, consider which letters from $A'$ can appear in $w$. Examining
the seventh, eighth, and ninth $\zset$-components (which are
unaffected by the actions of any of the $\symgroup{11}$-components),
shows that $w$ contains one letter $a$ or letter $e$, one letter $c$
or letter $h$, and one letter $e$ or letter $j$, and no other letter
from $A'$. For the product of the $\symgroup{11}$-components to be
$\id$, the letters from $A'$ in $w$ must then be $a,c,e$ or $f,h,j$
(in some order).

Consider these two cases separately:
\begin{enumerate}

\item Suppose first that the letters from $A'$ in $w$ are
  $a,c,e$. Since the $\symgroup{11}$-components of
  $\elt{a},\elt{c},\elt{e}$ do not affect the fifth and sixth
  $\zset$-components, and since these are both $0$ in $s$, $w$ cannot
  contain letters $g$ or $h$. So $w$ is a rearrangement of
  $aceb^\gamma d^\delta$ for some $\gamma,\delta \in
  \nset\cup\{0\}$. Now, in $w$ the letter $a$ must precede the letter
  $c$, for otherwise the third $\zset$-component of $\elt{w}$ would be
  non-zero. Similarly, $c$ must precede $e$, for otherwise the fourth
  $\zset$-component of $\elt{w}$ would be non-zero. The letters $b$
  must all lie between $a$ and $c$, for otherwise the third
  $\zset$-component of $\elt{w}$ would be non-zero, and similarly the
  letters $d$ must all lie between $c$ and $e$, for otherwise the
  fourth $\zset$-component of $\elt{w}$ would be non-zero. So $w =
  ab^\gamma cd^\delta e$. Examining the first and second
  $\zset$-components forces $\gamma = \alpha$ and $\delta=\beta$. So
  if the letters from $A'$ in $w$ are $a,c,e$, then $w = ab^\alpha
  cd^\beta e$.

\item Suppose now that the letters from $A'$ in $w$ are $f,h,j$. Since
  the $\symgroup{11}$-components of $\elt{f},\elt{h},\elt{j}$ do not
  affect the third and fourth $\zset$-components, and since these are
  both $0$ in $s$, $w$ cannot contain letters $b$ or $c$. So $w$ is a
  rearrangement of $fhjg^\gamma i^\delta$ for some $\gamma,\delta \in
  \nset\cup\{0\}$. Now, in $w$, the letter $f$ must precede the letter
  $h$, for otherwise the fifth $\zset$-component of $\elt{w}$ would be
  non-zero. Similarly, $h$ must precede $j$, for otherwise the sixth
  $\zset$-component of $\elt{w}$ would be non-zero. The letters $g$
  must all lie between $f$ and $h$, for otherwise the fifth
  $\zset$-component of $\elt{w}$ would be non-zero, and similarly the
  letters $i$ must all lie between $h$ and $j$, for otherwise the sixth
  $\zset$-component of $\elt{w}$ would be non-zero. So $w = fg^\gamma
  hi^\delta j$. Examining the first and second $\zset$-components
  forces $\gamma = \alpha$ and $\delta=\beta$. So if the letters from
  $A'$ in $w$ are $f,h,j$, then $w = fg^\alpha hi^\beta j$. In this case,
\[
\elt{w} = [\id,\;(\alpha+2,\beta+2,0,0,0,0,1,1,1,0,\alpha-\beta)].
\]
By \eqref{eq:virtabel1}, this forces $\alpha = \beta$, and $w =
fg^\alpha hi^\alpha j$.
\end{enumerate}
So if $\alpha \neq \beta$, only the first case holds and $w =
ab^\alpha cd^\beta e$. If, on the other hand, $\alpha = \beta$, then
both cases can hold and $w$ is either $ab^\alpha cd^\alpha e$ or
$fg^\alpha hi^\alpha j$.

Parallel reasoning shows that if $\alpha \neq \beta$, the element
represented by $fg^\alpha hi^\beta j$ is represented by no other word
over $A$.
\end{proof}

By \fullref{Lemma}{lem:vabelreps}, the semigroup $S$ satisfies the
hypotheses of \fullref{Lemma}{lem:abcdefghijnonmarkov} and so is not Markov.
\end{example}

\section{Miscellaneous examples of Markov and non-Markov semigroups}
\label{sec:examples}

This section gathers miscellaneous examples to illustrate particular
aspects of the class of Markov semigroups.

First, here is an example of a non-Markov semigroup:

\begin{example}
Let $A = \{a,b,c,d\}$ and let
\[
\rel{S} = \{(ba,ab),(bc,aca), (acc,d)\} \cup \{(dx,d),(xd,d) : x \in A\}.
\]
The monoid presented by $\pres{A}{\rel{S}}$ does not admit a regular
language of unique representatives by
\cite[Example~4.6]{otto_infiniteconvergent}, and thus is not Markov.
\end{example}

Since free groups of finite rank are Markov (either by
\fullref{Proposition}{prop:rsmarkov} or as a corollary of
\cite[Proposition~9]{ghys_markov}) and indeed strongly Markov
(since they are hyperbolic; see
\cite[Th\'{e}or\`{e}me~13]{ghys_markov}), the
following example is worth noting:

\begin{example}
The free inverse monoid of rank $1$ is not Markov, because it admits
no regular language of unique normal forms over the generating set
\cite[Proof of Theorem~2.7]{cutting_normalforms}.
\end{example}

The Baumslag--Solitar groups play their customary r\^{o}le of being
pleasant and easy to understand but slightly eccentric. This is a
consequence of the following theorem of Groves:

\begin{theorem}[{\cite[Corollary in \S~1]{groves_minimallength}}]
\label{thm:bs}
There is no regular language of minimal-length representatives for the
Baumslag--Solitar groups
\[
\pres{a,t}{(t^{-1}at,a^p)},
\]
where $p > 1$ with respect to the alphabet $\{a,a^{-1},t,t^{-1}\}$.
\end{theorem}

[The original statement of this result by Groves is phrases in terms of
  minimal-length (unique) normal forms. However, the property of
  uniqueness is not used anywhere in the proof. Groves states the
  result in these terms because he places the result in the context of
  calculating growth series.]

\begin{example}
\label{ex:bs}
The Baumslag-Solitar group $\pres{a,t}{(t^{-1}at,a^2)}$ is presented
by the a confluent noetherian rewriting system
\cite[p.~156]{epstein_wordproc}, and is therefore Markov by
\fullref{Proposition}{prop:rsmarkov}. However, since it admits no
regular language of minimal-length representatives by
\fullref{Theorem}{thm:bs}, it is not strongly Markov. (However, it
does admit a one-counter language of minimal-length normal forms
\cite[\S\S~4--5]{elder_baumslagsolitar}.)
\end{example}

This example leads on to the following question:

\begin{question}
Is every one-relation semigroup Markov?
\end{question}

If every one-relation semigroup can be presented by a confluent
noetherian rewriting system (an open question, since it would imply a
solution to the world problem), this question would have a positive
answer by \fullref{Proposition}{prop:rsmarkov}.

A robustly Markov monoid may not be residually finite:

\begin{example}
Let $A = \{a,b\}$ and let $\rel{R} = \{(ab^2,b)\}$. Then $(A,\rel{R})$
is a confluent noetherian rewriting system and so the monoid $M$
presented by $\pres{A}{\rel{R}}$ is Markov by
\fullref{Proposition}{prop:rsmarkov}. This monoid $M$ is known to be
non-residually finite~\cite{lallement_onerelmon}.
\end{example}

A strongly Markov monoid may not be finitely presented:

\begin{example}
Let $A = \{a,b,c,d,e,f\}$ and $\rel{R} = \{(ab^nc, de^nf) : n \in
\nset\}$. Let $M$ be the monoid presented by $\pres{A}{\rel{R}}$. Then
$M$ is not finitely presented since no relation in $\rel{R}$ can be
deduced from the others. But $M$ is strongly Markov: since every
generators in $\elt{A}$ is indecomposable, any alphabet representing a
generating set for $M$ must contain a subalphabet representing
$\elt{A}$; thus $A^* - A^*ab^*cA^*$ is a robust Markov language for
$M$ over any alphabet representing a generating set.
\end{example}

This example suggests the following question:

\begin{question}
Does there exist a strongly Markov group that is not finitely
presented? If not, does there exists a non-finitely presented Markov
or robustly Markov group? [The authors conjecture that the answers to
  these questions are both yes, for intuition suggests that a Markov
  or robust Markov language does not impose enough structure on a
  group to guarantee finite presentability.]
\end{question}

The following easy example shows that it is possible for a robustly
Markov monoid to have unsolvable word problem:

\begin{example}
\label{ex:wpunsolvable}
Let $I$ be a non-recursive subset of $\nset$. Let $A = \{a,b,c,x,y\}$ and
\[
\rel{R} = \{(ab^\alpha c,x) : \alpha \in I\} \cup \{(ab^\alpha c,y) : \alpha \notin I\}.
\]
The rewriting system $(A,\rel{R})$ is confluent because left-hand
sides of rules in $\rel{R}$ overlap only when they are identical. It
is noetherian because it is length-reducing. The language of left-hand
sides of rules in $\rel{R}$ is
\[
\{ab^\alpha c : \alpha \in I\} \cup \{ab^\alpha c : \alpha \notin I\} = ab^*c
\]
and so is regular. By \fullref{Proposition}{prop:rsmarkov}, $L$ is a
robust Markov language for the monoid presented by
$\pres{A}{\rel{R}}$.

However, this monoid does not have solvable word problem, since
$ab^\alpha c$ and $x$ represent the same element of the semigroup if
and only if $\alpha \in I$. But membership of $I$ is undecidable since
$I$ is non-recursive.
\end{example}

However, a \emph{finitely presented} Markov semigroup will have
soluble word problem, as does any finitely presented semigroup that
admits a recursively enumerable language of unique representatives
\cite[Theorem~1.5]{cutting_normalforms}.

\section{Hyperbolicity \& Automaticity}
\label{sec:hyperbolicity}

Ghys et al.\ proved that hyperbolic groups are Markov using a direct
approach~\cite[\S 3]{ghys_markov}. It also follows using the machinery
of automatic groups: over any generating set, the language of
geodesics is regular and forms part of a prefix-closed automatic
structure \cite[Theorem~3.4.5]{epstein_wordproc}, and the construction
of an automatic structure with uniqueness
\cite[Theorem~2.5.1]{epstein_wordproc} preserves prefix-closure when
applied in this particular case (although \emph{not} in the general
case).

Hyperbolicity can be generalized from groups to semigroups in either a
geometric or linguistic sense. The latter generalization, which is
termed \defterm{word-hyperbolicity}, is due to Duncan \& Gilman
\cite{duncan_hyperbolic}. It informally says that a semigroup is
word-hyperbolic if it admits a regular language of representatives
such that the multiplication table in terms of these representatives
is a context-free language. 

\begin{definition}
\label{def:wordhyp}
A \defterm{word-hyperbolic structure} for a semigroup $S$ is a pair
$(A,L)$, where $A$ is a finite alphabet representing a generating set
for $S$ and $L$ is a regular language over $A$ such that $\elt{L} = S$
and the language
\[
M(L) = \{u\#_1v\#_2w^\rev : u,v,w \in L \land \elt{uv} = \elt{w}\}
\]
(where $\#_1$ and $\#_2$ are new symbols not in $A$) is context-free.

A semigroup is \defterm{word-hyperbolic} if it admits a
word-hyperbolic structure.
\end{definition}

A group is word-hyperbolic in the sense of
\fullref{Definition}{def:wordhyp} if and only if it is hyperbolic in
the sense of Gromov \cite[Corollary~4.3]{duncan_hyperbolic}. For
further background information on word-hyperbolic semigroups, see
\cite{duncan_hyperbolic,hoffmann_relatives}.

The following example is taken from
\cite[Example~4.2]{cm_wordhypunique}:

\begin{example}
\label{ex:wordhypnotmarkov}
Let $A = \{a,b,c,d\}$ and let $\rel{R} = \{(ab^\alpha c^\alpha
d,\emptyword) : \alpha \in \nset\}$. Let $M$ be the monoid presented
by $\pres{A}{\rel{R}}$. Since the rewriting system $(A,\rel{R})$ is
context-free, $M$ is word-hyperbolic by
\cite[Theorem~3.1]{cm_wordhypunique}. The reasoning in
\cite[Example~4.2]{cm_wordhypunique} shows that it does not admit a
regular language of unique normal forms over any generating set, and
so in particular cannot be Markov by
\fullref{Proposition}{prop:markovchangegen}.
\end{example}

Thus word-hyperbolic monoids are not in general Markov. Moreover if
the regularity condition on the left-hand sides of rewriting rules in
\fullref{Proposition}{prop:rsmarkov} is weakened to being context-free
(or even just to being one-counter), then the semigroups or monoids
thus presented are not Markov in general.

\fullref{Example}{ex:wordhypnotmarkov} is not finitely presented, and
it does not admit a word-hyperbolic structure with uniqueness
\cite[Example~4.2]{cm_wordhypunique}. This provokes the following questions:

\begin{question}
Does there exist a non-Markov finitely presented word-hyperbolic
monoid?
\end{question}

\begin{question}
Does there exist a non-Markov monoid that admits a word-hyperbolic
structure with uniqueness? 
\end{question}

Since satisfying a linear isoperimetric inequality is one of several
equivalent characterizations of hyperbolic groups (see, for example,
\cite[Ch.~1]{alonso_wordhyperbolic}), the following question is of
interest:

\begin{question}
Does there exist a non-Markov semigroup with linear isoperimetric
inequality?
\end{question}

Markov groups are not in general automatic, since all polycyclic
groups are Markov \cite[Corollaire~11]{ghys_markov}, but a
nilpotent group that is not virtually abelian cannot be automatic
\cite[Theorem~8.2.8]{epstein_wordproc}.

\begin{question}
Are automatic semigroups Markov? (Note that, unlike the situation for
groups, an automatic semigroup need not be word-hyperbolic.) This
question relates to the long-standing open question of whether an
automatic semigroup or group admits a prefix-closed automatic
structure with uniqueness \cite[Open
  Question~2.5.10]{epstein_wordproc}. Admitting such an automatic
structure entails being Markov.
\end{question}

\section{Adjoining an identity or zero}
\label{sec:adjoin}

This section and those that follow examines the interaction of the
classes of Markov, robustly Markov, and strongly Markov semigroups
with various semigroup constructions. The main questions are whether
these classes of semigroups are closed under a particular
construction, and whether the semigroup resulting from such a
construction being Markov, robustly Markov, or strongly Markov implies
that the original semigroup is (or the original semigroups are)
Markov, robustly Markov, or strongly Markov.

Arguably the simplest semigroup construction are the adjoining of an
identity or zero, and it is reassuring that both questions have
positive answers for these constructions:

\begin{proposition}
\label{prop:adjone}
Let $S$ be a semigroup. Then:
\begin{enumerate}
\item $S$ is Markov if and only if $S^1$ is Markov.
\item $S$ is robustly Markov if and only if $S^1$ is robustly Markov.
\item $S$ is strongly Markov if and only if $S^1$ is strongly Markov.
\end{enumerate}
\end{proposition}

\begin{proof}
Let $A$ be a finite alphabet representing a semigroup generating set
for $S$. Let $1$ be a new symbol not in $A$ representing the adjoined
identity of $S^1$.

Let $L$ be a semigroup Markov language for $S$ with respect to
$A$. Then $L$ is regular, $+$-prefix-closed, and maps bijectively onto
$S$. Let $K = L \cup \{1\}$. Then $K$ is regular, $+$-prefix-closed,
and maps bijectively onto $S^1$. Thus $K$ is a semigroup Markov
language for $S^1$.

Furthermore, if $L$ is a robust semigroup Markov language, then so is
$K$, since $1$ is the unique shortest word representing the
adjoined identity, and the natural lengths of elements in $S$ over $A$
and over $A\cup\{1\}$ are equal.

\smallskip
Now let $L$ be a semigroup Markov language
for $S^1$ over an alphabet $B$ representing some generating set for
$S^1$. Now, $B$ must be of the form $A \cup \{1\}$, where $1$
represents the adjoined identity and $A$ represents a generating set
for $S$, since no product of elements of $S$ equals the adjoined
identity.

Suppose some $w \in L$ contains the symbol $1$. Then $w = w'1w''$ and
so $w'$ and $w'1$ represent the same element of $S^1$, unless $w'$ is
the empty word, which is not a member of the semigroup Markov
language $L$. So such a word $w$ can only contain a single instance
of the symbol $1$, and it must be the first symbol of $w$. (If $L$ is
a robust semigroup Markov language, the only such word is $w=1$, since
otherwise $w'w''$ would be a shorter word representing $\elt{w}$, as
in the proof of \fullref{Proposition}{prop:stronglymarkovsgmon}.)

Let
\[
K = \Big(\big(L - \{1\}\big) - 1A^*\Big) \cup \{u \in A^+ : 1u \in L\}.
\]
Arguing as in the proof of \fullref{Proposition}{prop:markovsgmon}, it
follows that $K$ is $+$-prefix-closed, is regular, and contains a
unique representative for each element of $S$. Thus $K$ is a semigroup
Markov language for $S$ over the alphabet $A$.

Furthermore, if $L$ is a robust semigroup Markov language, the only
word in $L$ containing the symbol $1$ is the word $1$ itself, so in this case
\[
K = L - \{1\}.
\]

\smallskip
From these arguments, it follows that $S$ is Markov if and only if
$S^1$ is Markov and that $S$ is robustly Markov if and only if $S^1$ is
robustly Markov. From the arbitrary choice of generating sets, and the
fact that any alphabet representing a generating set for $S^1$ must be
of the form $A \cup \{1\}$, where $1$ represents the adjoined identity
and $A$ represents a generating set for $S$, it follows that $S$ is
strongly Markov if and only if $S^1$ is strongly Markov.
\end{proof}

\begin{proposition}
\label{prop:adjzero}
Let $S$ be a semigroup. Then:
\begin{enumerate}
\item $S$ is Markov if and only if $S^0$ is Markov.
\item $S$ is robustly Markov if and only if $S^0$ is robustly Markov.
\item $S$ is strongly Markov if and only if $S^0$ is strongly Markov.
\end{enumerate}
\end{proposition}

\begin{proof}
By reasoning parallel to the proof \fullref{Proposition}{prop:adjone},
substituting $0$ for $1$ and $S^0$ for $S^1$ as appropriate, it
follows that if $L$ is a [robust] Markov language for $S$,
then $L\cup \{0\}$ is a [robust] Markov language for $L$.

\smallskip
Now let $L$ be a Markov language
for $S^0$ over an alphabet $B$ representing some generating set for
$S^0$. Now, $B$ must be of the form $A \cup \{0\}$, where $0$
represents the adjoined zero and $A$ represents a generating set
for $S$, since no product of elements of $S$ equals the adjoined
zero.

Suppose some $w \in L$ contains the symbol $0$, with $w =
w'0w''$. Then $w'0$ and $w$ both represent the zero of the semigroup,
which contradicts the uniqueness of representatives in $L$ unless
$w''$ is the empty word. So such a word $w$ can contain only a
single symbol $0$, and this must be the last letter of the word. (If
$L$ is a robust Markov language, the only such word is $w=0$
since this is the unique shortest word over $A \cup \{0\}$
representing the adjoined zero.) Notice that there can only be one
such word, since any other word containing the symbol $0$ would also
represent the adjoined zero. So $L$ contains a unique word $w =
w'0$ containing the symbol $0$, and this word is not the prefix of any
other word in $L$.

Let $K = L - \{w'0\}$. Then $K$ is $+$-prefix-closed (since $w'0$ is
not a prefix of any other word in $L$), is regular, and contains a
unique representative for each element of $S$. Finally, $K \subseteq
A^+$ by the observation at the end of the last paragraph. Thus $K$ is
a Markov language for $S$ over the alphabet $A$.

\smallskip
From these arguments, it follows that $S$ is Markov if and
only if $S^0$ is Markov and that $S$ is robustly Markov if and
only if $S^0$ is robustly Markov. From the arbitrary choice of
generating sets, and the fact that any alphabet representing a
generating set for $S^0$ must be of the form $A \cup \{0\}$, it
follows that $S$ is strongly Markov if and only if $S^1$ is
strongly Markov.
\end{proof}

\section{Direct products}
\label{sec:dirprod}

The class of Markov groups is closed under direct products, as a
special case of the fact that an extension of one Markov group by
another is also Markov \cite[Proposition~10]{ghys_markov}. For
monoids, the result is also positive:

\begin{theorem}
\label{thm:dirprodmon}
\begin{enumerate}
\item If $M$ and $N$ are Markov monoids, then $M \times N$ is a Markov
monoid. 
\item If $M$ and $N$ are robust Markov monoids, then $M \times N$ is a
  robust Markov monoid.
\end{enumerate}
\end{theorem}

\begin{proof}
\begin {enumerate}

\item Let $A$ and $B$ be finite alphabets representing monoid generating
sets for $M$ and $N$ with representation maps $\phi_A : A \to M$ and
$\phi_B : B \to N$, respectively, and let $K$ and $L$ be monoid Markov
languages over $A$ and $B$ for $M$ and $N$, respectively. Then $H =
KL$ is prefix-closed, regular, and maps bijectively onto $M \times N$
under the representation map $\phi : A \cup B \to M \times N$ defined
by $a \mapsto (a\phi_A,1_N)$ and $b \mapsto (1_M,b\phi_B)$.

\item Proceed as in the previous part, but with $K$ and $L$ being
  robust Markov languages. Then $KL$ is a robust Markov language for
  $M \times N$ since (with respect to the representation map $\phi$)
  $\lambda_{A \cup B}(\elt{uv}) = \lambda_A(\elt{u}) +
  \lambda_B(\elt{v})$ for all $u \in K$ and $v \in L$.\qedhere
\end {enumerate}
\end{proof}

However, for semigroups the situation is obscure. First of all, a
direct product of finitely generated semigroups is not necessarily
finitely generated. For example, the direct product of two copies of
the natural numbers $\nset$ (excluding $0$) is not finitely
generated. (Notice that $\nset$ is strongly Markov.) Even when the
direct product is finitely generated, the relationship of a finite
generating set to the finite generating sets of the direct factors is
complex; see the discussion in \cite[\S~2]{robertson_dirprod}. It is
possible to prove that a direct product of a Markov semigroup and a
finite semigroup is Markov if it is finitely generated
(\fullref{Theorem}{thm:markovdirprodfinite} below). The general idea
of the proof is similar to that used by Campbell et al.\ to prove the
analogous result for automatic semigroups
\cite[Theorem~1.1(ii)]{campbell_dirprodautsg}, but more sophisticated
reasoning is required here to ensure that prefix-closure and
uniqueness are preserved. However, the issue of prefix-closure seems
to make it impossible to adapt and strengthen the idea used by
Campbell et al.\ for direct products of infinite semigroups. An
entirely new approach may be required in this case.

\begin{theorem}
\label{thm:markovdirprodfinite}
Let $S$ be a Markov semigroup and let $T$ be finite. Then $S \times T$
is a Markov semigroup if and only if it is finitely generated.
\end{theorem}

\begin{proof}
One direction of the result is trivial: if $S \times T$ is a Markov
semigroup, then by definition it is finitely generated.

Suppose that $S \times T$ is finitely generated. Then by
\cite[Lemma~2.3]{robertson_dirprod}, the finite semigroup $T$ is
such that $T^2 = T$.

Since $S$ is a Markov semigroup, it admits a Markov language $L$ over
some finite alphabet $A$ representing a generating set for $S$.

Let $B$ be a finite alphabet in bijection with $T$.  Since $T^2 = T$,
it follows that, $T^n = T$ for all $n \in \nset$ and so for any $t \in T$ and $n \in
\nset$, there is word of length $n$ over $B$ representing $t$. Let
\[
R = \bigl\{(u,v) : u,v \in B^+, |u| = |v|, \elt{u} = \elt{v}\bigr\};
\]
notice that $R$ is a synchronous rational relation. Let
$\lex\sqsubset$ be the lexicographic ordering on $B^+$ based on some
total ordering of $B$. Then
\[
R' = \bigl\{ u : (\forall v \in A^*)((u,v) \in R \implies u \lex\sqsubset v)\bigr\}.
\]
The language $R'$ contains exactly one (lexicographically minimal)
representative of each length for each element of $T$. Furthermore, the
language $R'$ is $+$-prefix-closed, for if $u$ is not
$\lex\sqsubset$-minimal amongst words of length $|u|$ representing
$\elt{u}$, then for any $a \in A$, the word $ua$ is not
$\lex\sqsubset$-minimal amongst words of length $|ua|$ representing
$\elt{ua}$.

Define 
\[
\delta : \bigcup_{n=0}^\infty \big(A^n \times B^n\big) \to (A \times B)^*
\]
(so that $(u,v)\delta$ is defined when $u \in A^*$ and $v \in B^*$
have equal length) by
\[
(a_1a_2\cdots a_n,b_1b_2\cdots b_n) \mapsto (a_1,b_1)(a_2,b_2)\cdots(a_n,b_n),
\]
where $a_i \in A$, $b_i \in B$.

Let $K = \{(w,u) : w \in L, u\in R', |w|=|u|\}$. Then $K\delta$ is a
regular language over $A \times B$. Since both $L$ and $R'$ are
$+$-prefix-closed, so is $K\delta$.

Now let $(s,t) \in S \times T$. Then since $L$ maps onto $S$, there is
a word $w \in L$ with $\elt{w} = s$. There is a word $u'$ of length
$|w|$ over $B$ such that $\elt{u'} = t$. Let $u$ be the
$\lex\sqsubset$-minimal such word. Then $|u| = |w|$ and so $(w,u)\in
K$ and so $(w,u)\delta \in K\delta$ represents $(s,t)$. So $K\delta$
maps onto $S \times T$.

Now suppose $(w,u)\delta, (w',u')\delta \in K\delta$ represent the
same element of $S \times T$. Then $\elt{w} = \elt{w'}$ and $\elt{u} =
\elt{u'}$. Since $L$ is a Markov language for $S$, it maps bijectively
onto $S$ and so $w = w'$. In particular, $|w| = |w'|$, and so $|u| =
|u'|$ by the definition of $K$. Since $\elt{u} = \elt{u'}$ and $|u| =
|u'|$, and $R'$ contains exactly one representative of $\elt{u}$ of
length $|u|$, it follows that $u = u'$. Hence $(w,u)\delta =
(w',u')$. Therefore $K\delta$ maps bijectively onto $S \times T$.

Thus $K\delta$ is a Markov language for $S \times T$ and so $S \times
T$ is a Markov semigroup.
\end{proof}

\begin{theorem}
\label{thm:robustmarkovdirprodfinite}
Let $S$ be a robustly Markov semigroup and let $T$ be finite. Then $S \times T$
is a robustly Markov semigroup if and only if it is finitely generated.
\end{theorem}

\begin{proof}
Proceed as in the proof of
\fullref{Theorem}{thm:robustmarkovdirprodfinite}, with $L$ being a
robust Markov language for $S$. Since $\lambda_B(t) = 1$ for all $t
\in T$, it follows that $\lambda_{(A\times B)\delta}(s,t) =
\lambda_A(s)$. So, by its construction, $K\delta$ is a robust Markov
language for $S \times T$.
\end{proof}

The corresponding result for being strongly Markov is still open:

\begin{question}
Let $S$ be strong Markov and $T$ finite. If $S\times T$ is finitely generated, is it strongly Markov?
\end{question}

We conjecture that the answer to this question is `yes', but probably
requires more complex reasoning than in the proofs of
\fullref{Theorems}{thm:markovdirprodfinite} and
\ref{thm:robustmarkovdirprodfinite}, because the generating set for $S
\times T$ may not project onto $T$, which complicates the relationship
between minimal lengths of representatives of elements of $S \times T$
and $T$.

As remarked above, the following question is open:

\begin{question}
Let $S$ and $T$ be Markov. If $S\times T$ is finitely generated, is it Markov?
\end{question}

The following question also arises:

\begin{question}
Is it true that whenever $S\times T$ is Markov, then both factors
$S$ and $T$ are Markov?
\end{question}

The answer to this question may shed light on the long-standing open
question of whether direct factors of automatic groups, monoids, or
semigroups must themselves be automatic (see \cite[Open
  Question~4.1.2]{epstein_wordproc} and
\cite[Question~6.6]{campbell_autsg}).

\section{Free products}
\label{sec:freeprod}

\begin{theorem}
The class of Markov monoids is closed under forming (monoid)
free products.
\end{theorem}

\begin{proof}
The proof for groups generalizes
directly~\cite[Proposition~9]{ghys_markov}.
\end{proof}

\begin{theorem}
The class of Markov semigroups, the class of robustly Markov
semigroups, and the class of strongly Markov semigroups are all closed
under forming (semigroup) free products.
\end{theorem}

\begin{proof}
Let $S$ and $T$ be Markov semigroups. Let $K \subseteq A^+$ and $L
\subseteq B^+$ be semigroup Markov languages for $S$ and $T$, respectively. Let
\[
M = (KL)^+ \cup (KL)^*K \cup (LK)^+ \cup (LK)^*K.
\]
Since the languages $K$ and $L$ are prefix-closed and regular, so is
the language $M$. Any element of the free product $S \ast T$ has a
unique representation as an alternating product of elements of $S$ and
$T$. That is $S \ast T$ is the disjoint union of
\begin{align*}
X_1 &= \{s_1t_1\cdots s_nt_n : s_i \in S, t_i\in T, n \in \nset\}, \\
X_2 &= \{s_1t_1\cdots s_nt_ns_{n+1} : s_i \in S, t_i\in T, n \in \nset\cup\{0\}\}, \\
X_3 &= \{t_1s_1\cdots t_ns_n : s_i \in S, t_i\in T, n \in \nset\}, \\
X_4 &= \{t_1s_1\cdots t_ns_nt_{n+1} : s_i \in S, t_i\in T, n \in \nset\cup\{0\}\}.
\end{align*}
Since the languages $K$ and $L$ do not contain the empty word, every
element of $X_1$ (respectively $X_2,X_3,X_4$) has a unique
representative in $(KL)^+$ (respectively $(KL)^*K, (LK)^+,
(LK^*K$). So every element of $S \ast T$ has a unique representative
in $M$. So $M$ is a Markov language for $S \ast T$.

Following the same reasoning with $S$ and $T$ being robustly Markov
semigroups and $K$ and $L$ being robust Markov languages shows that $M$ is a robust Markov language for $S \ast T$,
since,
\[
\lambda_{A \cup B}(s_1t_1\cdots s_nt_n) = \sum_{i=1}^n \big(\lambda_A(s_i) + \lambda_B(t_i)\big),
\]
and similarly for alternating products in $X_2 \cup X_3 \cup X_4$.

Finally, suppose that $S$ and $T$ are strongly Markov semigroups. Let
$C$ be a finite alphabet representing a generating set for $S \ast
T$. Since $S \ast T$ is a \emph{semigroup} free product, $C$ contains
subalphabets $A$ and $B$ representing generating sets for $S$ and $T$
respectively. Since $S$ and $T$ are strongly Markov semigroups, there
exist robust Markov languages $K \subseteq A^+$ and $L \subseteq B^+$
for $S$ and $T$ respectively. Thus, by the preceding paragraph, $M
\subseteq (A \cup B)^+ \subseteq C^+$ is a robust Markov language for
$S \ast T$. Since $C$ was arbitrary, $S \ast T$ is strongly Markov.
\end{proof}

\section{Finite-index extensions and subsemigroups}
\label{sec:finind}

Many properties of groups are known to be preserved under passing from
groups to finite-index extensions and subgroups; for example, finite
generation and presentability. For semigroups, the most well-known
notion of index is the Rees index: if $T$ is subsemigroup of a semigroup
$S$, then $T$ has finite index in $S$ if $S-T$ is finite. Many
properties of semigroups are known to be preserved on passing to
finite Rees index extensions and subsemigroups; for example, finite
generation~\cite[Theorem~1.1]{ruskuc_largesubsemigroups}, finite
presentability~\cite[Theorem~1.3]{ruskuc_largesubsemigroups}, and
automaticity~\cite[Theorem~1.1]{hoffmann_autfinrees}. The following
result fits this pattern:

\begin{theorem}
The class of Markov semigroups is closed under forming finite Rees
index extensions and subsemigroups.
\end{theorem}

\begin{proof}
Let $S$ be a semigroup and let $T$ be a finite Rees index
subsemigroup of $S$.

Suppose that $T$ is Markov and that $L$ is a Markov language for $T$
over some finite alphabet $A$ representing a generating set for
$T$. Let $B$ be an alphabet in bijection with $S - T$; then $B$ is
finite since $T$ has finite Rees index in $S$. Without loss of
generality, assume that $B$ and $A$ are disjoint. Then $L \cup B$
is a Markov language for $S$.

Now suppose that $S$ admits a Markov language $L$ over an alphabet
$A$.

Define
\[
L(A,T) = \{w \in A^+ : \elt{w} \in T\}.
\]
Let $C$ be an alphabet of unique representatives for $S-T$.  For any
word $w \in A^* - L(A,T)$, let $\rep{w}$ be the unique element of $C
\cup \{\emptyword\}$ representing $\elt{w}$, or $\emptyword$ if $w =
\emptyword$.

Define the alphabet
\[
D = \{d_{\rho, a,\sigma} : \rho,\sigma \in C \cup \{\emptyword\}, a \in A,
a\sigma \in L(A,T) \land \rho a\sigma \in L(A,T)\},
\]
and let it represent elements of $T$ as follows: 
\[
\elt{d_{\rho, a,\sigma}} = \elt{\rho a\sigma}.
\]
Notice that if $A$ is finite, $D$ too must be finite.

Let $\rel{R} \subseteq A^+ \times D^+$ be the relation consisting of pairs
\[
\big(w_{n+1}a_nw_na_{n-1}w_{n-1}\cdots a_2w_2a_1w_1,d_{\rep{w_{n+1}},a,\rep{w_n}}d_{\emptyword,a,\rep{w_{n-1}}}\cdots d_{\emptyword,a,\rep{w_2}}d_{\emptyword,a,\rep{w_1}}\big)
\]
where the left-hand side lies in $L(A,T)$ and the factorization of the
left-hand side is obtained in the following way: start by letting the
left-hand side be $w'_0$; a partial factorization
\[
w'_{i+1}a_iw_i\cdots a_1w_1
\]
is complete if $w'_{i+1} \notin L(A,T)$; if on the other hand
$w'_{i+1} \in L(A,T)$ set $a_{i+1}w_{i+1}$ to be the shortest suffix of
$w'_{i+1}$ lying in $L(A,T)$ and let $w'_{i+2}$ be the remainder of
$w'_{i+1}$.

Notice that if $(w,u) \in \rel{R}$ then $\elt{w} = \elt{u}$ by the
definition of how the alphabet $D$ represents element of $T$, and that
each word $w$ determines a unique word $u$ such that $(w,u) \in
\rel{R}$.

\begin{lemma}
\label{lem:reesindexrational}
The relation $\rel{R}$ is rational.
\end{lemma}

\begin{proof}
It is easier to explain a how a two-tape finite state automaton
$\fsa{A}$ can recognize $\rel{R}$ when reading from right-to-left;
since the class of rational relations is closed under reversal, it
will then follow that $\rel{R}$ is rational.

By the dual of \cite[Theorem~4.3]{ruskuc_syntactic}, $S$ admits a left
congruence $\Lambda$ of finite index (that is, having finitely many
equivalence classes) contained within $(T \times T) \cup
\Delta_{S-T}$, where $\Delta_{S-T}$ is the diagonal relation on $S-T$
(that is, $\{(s,s) : s \in S - T\}$).

Imagine the automaton $\fsa{A}$ reading letters from $A$ from its
left-hand input tape and outputting symbols from $D$ on its right-hand
tape. Suppose the content of its left-hand tape is $w$. As it reads
symbols from $w$ (moving from right to left along the tape), it keeps
track of the $\Lambda$-class of the element represented by the suffix
of $w$ read so far. (This is possible because $\Lambda$ is a left
congruence with only finitely many equivalence classes.) In
particular, $\fsa{A}$ knows whether the element represented by the
suffix read so far lies in $T$ (or equivalently, whether the suffix
read so far lies in $L(A,T)$), or, if the element so represented lies
in $S-T$, which letter of $C \cup \{\emptyword\}$ represents it. When
$\fsa{A}$ reads a symbol $a$ such that the suffix read so far --- say
$aw'$ --- lies in $L(A,T)$, it non-deterministically chooses one of
two actions:
\begin{enumerate}

\item It outputs $d_{\emptyword,a,\rep{w'}}$, resets its store of the
  suffix read so far to $\emptyword$, and continues to read from its
  left-hand tape.

\item It outputs $d_{c,a,\rep{w'}}$, where $c$ is a
  non-deterministically chosen element of $C \cup \{\emptyword\}$,
  then reads the remainder $v$ of its left-hand tape and accepts if
  and only if $\rep{v} = c$. (Notice that this is the only way that $\fsa{A}$ can accept.)

\end{enumerate}
By induction on the subscripts of the letters $a_i$, the automaton
$\fsa{A}$ can accept only by outputting letters
$d_{\emptyword,a,\rep{w_i}}$ immediately after reading the suffix
$a_iw_i\cdots a_1w_1$ and the letter $d_{\rep{w_{n+1}},a_n,\rep{w_n}}$
  immediately after reading $a_nw_n\cdots a_1w_1$, and can accept only
  when $w_{n+1} \notin L(A,T)$. So $\fsa{A}$ recognizes $\rel{R}$,
  reading from left-to-right.
\end{proof}

By \fullref{Lemma}{lem:reesindexrational},
\[
K = L \circ \rel{R} = \big\{u \in D^*: (\exists v \in L)\big((u,v) \in \rel{R}\big)\big\}.
\]
is regular. Since the set of left-hand sides of elements of $\rel{R}$ is
$L(A,T)$, the language $K$ maps onto $T$.

Suppose $u_1,u_2 \in K$ are such that $\elt{u_1} = \elt{u_2}$. Let
$w_1,w_2 \in L$ be such that $(w_1,u_1), (w_2,u_2) \in \rel{R}$. Since
$L$ maps bijectively onto $S$ and $\elt{w_1} = \elt{u_1} = \elt{u_2} =
\elt{w_2}$, the words $w_1$ and $w_2$ must be identical. Since every
$w \in L(A,T)$ determines a unique $u \in D^+$ with $(w,u) \in
\rel{R}$, it follows that $u_1$ and $u_2$ are identical. So $K$ maps
bijectively onto $T$.

Finally, let $u \in K$ with $|u| \geq 2$. Then $u =
d_{c_{n+1},a_n,c_n}\cdots
d_{\emptyword,a_2,c_2}d_{\emptyword,a_1,c_1}$, with $n \geq 2$. Then
there is some word $w \in L$ with $(w,u) \in \rel{R}$. By the
definition of $\rel{R}$, the word $w$ factorizes as
$w_{n+1}a_nw_n\cdots a_2w_2a_1w_1 \in L$ with $\rep{w_i} = c_i$, and
$a_1w_1, a_2w_2, \ldots, w_{n+1}a_nw_n \in L(A,T)$.

Since $L$ is prefix-closed, $w_{n+1}a_nw_n\cdots a_2w_2 \in L$. Since
$a_2w_2, \ldots, w_{n+1}a_nw_n \in L(A,T)$, it follows that
$w_{n+1}a_nw_n\cdots a_2w_2 \in L(A,T)$. So, by the definition of
$\rel{R}$, it follows that $d_{c_{n+1},a_n,c_n}\cdots
d_{\emptyword,a_2,c_2} \in K$.

This shows that $K$ is closed under taking longest proper non-empty
prefixes. By induction, $K$ is $+$-prefix-closed. Hence $K$ is a
Markov language for $T$.
\end{proof}

However, the Rees index has the disadvantage that is does not
generalize the group index. This motivated Gray \& Ru\v{s}kuc
\cite{gray_green1} to develop the notion of Green index, which does
generalize the group index. The definition and only the necessary
properties of the Green index and related topics are given here; the
reader is referred to \cite[\S~1]{gray_green1} for further details.

\begin{definition}
\label{def:greenindex}
Let $S$ be a semigroup and let $T$ be a subsemigroup of $S$. The
$T$-relative Green's relations $\rel{R}^T$, $\rel{L}^T$, and
$\rel{H}^T$ are defined on $S$ as follows: for $x,y \in S$,
\begin{align*}
x\;\rel{R}^T\;y &\iff xT^1 = yT^1 \\
x\;\rel{L}^T\;y &\iff T^1x = T^1y \\
x\;\rel{H}^T\;y & \iff x\;\rel{R}^T\;y \land x\;\rel{L}^T\;y;
\end{align*}
these are equivalence relations~\cite[\S~1]{gray_green1}. The
$T$-relative $\rel{R}^T$-, $\rel{L}^T$-, and $\rel{H}^T$-classes (that
is, the equivalence classes of these relations) respect $T$, in the
sense that each such class lies either wholly in $T$ or wholly in
$S-T$.

The \defterm{Green index} of $T$ in $S$ is defined to be one more than
the number of $\rel{H}^T$-classes in $S-T$.
\end{definition}

Several properties are known to
be preserved under passing to finite Green index extensions and
subsemigroups, such as finite generation~\cite[Theorems~4.1 \&
  4.3]{cgr_greenindex}, others are known to hold on passing to finite
Green index subsemigroups and not on passing to finite Green index
extensions, such as automaticity~\cite[Theorem~10.1 \&
  Example~10.3]{cgr_greenindex}. The following example shows that
neither the class of Markov semigroups nor the class of strongly
Markov semigroups is not closed under finite Green index
extensions. Indeed, a finite Green index extension of a strongly
Markov semigroup need not be Markov:

\begin{example}
Let $G$ a finitely generated infinite torsion group. Let $B$ be an
alphabet representing a generating set for $G$. Let $A$ be a finite
alphabet in bijection with $B$. Let $F$ be the free group with basis
$\elt{A}$. The bijection from $\elt{A}$ to $\elt{B}$ naturally extends
to a surjective homomorphism $\phi : F \to G$. Let $S$ be the strong
semilattice of groups $\mathcal{S}(F,G,\phi)$. (See
\cite[\S\S~4.1--4.2]{howie_fundamentals} for background on strong
semilattices of groups.)

The free group is hyperbolic and therefore strongly Markov. Moreover, $F$
is a finite Green index subsemigroup of $S$, with $S - F$ consisting
of the single $\rel{H}_F$-class $G$.

Suppose that $S$ is Markov. Then by
\fullref{Proposition}{prop:markovchangegen}, $S$ admits a regular
language of unique normal forms $L$ over the alphabet $A \cup B$. By
the definition of multiplication in a strong semilattice of monoids,
the words in $L$ representing elements of $G$ are precisely those that
include at least one letter $B$. That is, the language of words in $L$
representing elements of $G$ is $K = L - A^*$. Since $L$ is regular,
$K$ is also. Since $L$ maps bijectively onto $S$ and $K \subseteq L$,
it follows that $K$ maps bijectively onto $G$. So if each letter $a
\in A$ is interpreted as representing the element $\elt{a}\phi$ of
$G$, then $K$ is a regular language of unique normal forms for $G$.
However, $G$, as a finitely generated infinite torsion group, does not
admit a regular language of unique normal forms by the reasoning
in~\cite[Example~2.5.12]{epstein_wordproc}. This is a contradiction,
and so $S$ cannot be Markov.
\end{example}

This example is similar in spirit to examples showing that neither the
class of finitely presented semigroups nor the class of automatic
semigroups is not closed under forming finite Green index
extensions~\cite[Examples~6.5 \&~10.3]{cgr_greenindex}. However, with
an extra condition on the Sch\"utzenberger groups of the $T$-relative
$\rel{H}$-classes in the complement, a positive result does
hold. First of all, recall the definitions of Sch\"utzenberger groups:

\begin{definition}
Retain notation from \fullref{Definition}{def:greenindex}. Let $H$ be
an $\rel{H}_T$. Let $\Stab(H) = \{ t \in T^1 : Ht = H \}$ (the
\defterm{stabilizer} of $H$ in $T$), and define an equivalence
$\sigma(H)$ on $\Stab(H)$ by $(x,y) \in \sigma(H)$ if and only if $hx
= hy$ for all $h\in H$. Then $\sigma(H)$ is a congruence on $\Stab(H)$
and $\Stab(H)/\sigma(H)$ is a group, called the
\defterm{Sch\"utzenberger group} of the $\rel{H}_T$-class $H$ and
denoted $\Gamma(H)$.
\end{definition}

\begin{proposition}
\label{prop:fingreenindexup}
Let $S$ be a semigroup and $T$ a subsemigroup of $S$ of finite Green
index. Suppose that $T$ is Markov and that the Sch\"utzenberger group
of every $T$-relative $\rel{H}$-class in $S - T$ is Markov. Then $S$
is Markov.
\end{proposition}

\begin{proof}
Let $L$ be a semigroup Markov language for $T$ over some finite
alphabet $A$ representing a generating set for $T$ under the map $\phi
: A \to T$.  Since $T$ has finite Green index in $S$, there are
finitely many $T$-relative $H$-classes $H_1,\ldots,H_n$ in $S-T$. By
hypothesis, every Sch\"utzenberger group $\Gamma(H_i)$ admits a
semigroup Markov language $L_i$ over some finite alphabet $A_i$
representing a generating set for $\Gamma(H_i)$ under the map $\phi_i
: A_i \to \Gamma(H_i)$.  For brevity, let $\sigma_i = \sigma(H_i)$.

For each $i = 1,\ldots,n$, fix an element $h_i \in H_i$. For each $i =
1,\ldots,n$ and $a \in A_i$, fix elements $s_{i,a} \in \Stab(H_i)$
such that $a\phi_i = [s_{i,a}]_{\sigma_i}$.

Let $A_i'$ be a new alphabet in bijection with $A_i$ under the map
$\alpha_i : A_i \to A_i'$. (Without loss of generality, assume that
the alphabet $A$ and the various alphabets $A_i$ and $A_i'$ are
pairwise disjoint.)  Define a map $\psi_i : A_i \cup
A'_i \to S$ as follows:
\begin{equation}
\label{eq:greenindexuprep}
a\psi_i =
\begin{cases}
s_{i,a} & \text{if $a \in A_i$,} \\
h_is_{a,i} & \text{if $a \in A'_i$.}
\end{cases}
\end{equation}
Let
\[
L'_i = \big\{(a\alpha_i)u \in A'_iA_i^* : au \in L_i, a \in A_i\big\}.
\]
(So $L'_i$ is the language obtained from $L_i$ by taking each word in
$L_i \subseteq A_i^+$ and replacing its first letter with the
corresponding letter from $A'_i$.) Notice that since $L_i$ is regular
and $+$-prefix-closed, so is $L'_i$.

Since $\Gamma(H_i)$ acts regularly on $H_i$ via
\[
x \cdot [s]_{\sigma_i} = xs,
\]
it follows that for every $y \in H_i$ there is a unique element $[s]_{\sigma_i} \in
\Gamma(H_i)$ such that $h_i\cdot [s]_{\sigma_i} = y$. Thus it follows
from~\eqref{eq:greenindexuprep} and the fact that $L_i$ is a Markov
language for $\Gamma(H)$ that for every $y \in H_i$ there is a unique
$w \in L_i$ such that $h_i(w\phi_i) = y$. Hence,
by~\eqref{eq:greenindexuprep} and the definition of $L'_i$, for every
$y \in H_i$ there is a unique word $v \in L'_i$ with $v\psi_i =
y$. Thus $L'_i$ maps bijectively onto $H_i$.

Finally, let 
\[
K = L \cup \bigcup_{i=1}^n L'_i.
\]
Then $K$ is $+$-prefix-closed and regular. Define
\[
\psi : A \cup \bigcup_{i=1}^n \big(A_i \cup A'_i\big)  \to S, \qquad a\psi =
\begin{cases}
a\phi & \text{if $w \in A$,} \\
a\psi_i & \text{if $w \in A_i \cup A'_i$.}
\end{cases}
\]
Then $\phi$ maps $K$ bijectively onto $S$. Hence $K$ is a semigroup
Markov language for $L$ .
\end{proof}

\fullref{Proposition}{prop:fingreenindexup} parallels
\cite[Theorem~6.1]{cgr_greenindex}, which shows that if $T$ is a
finite Green index subsemigroup of $S$, and $T$ \emph{and} all the
Sch\"utzenberger groups of the $T$-relative $\rel{H}$-classes in $S-T$
are finitely presented, then $S$ is finitely presented. (As remarked
above, without the condition on the finite presentability, this result
does not hold.) This is in marked contrast to the situation for
automatic groups: even if $T$ and all the Sch\"utzenberger groups are
automatic, $S$ may not be automatic; see
\cite[Example~10.3]{cgr_greenindex}.

\begin{question}
Let $T$ be a subsemigroup of finite Green index in a semigroup $S$.
Let also $S$ be Markov. Is $T$ Markov?
\end{question}

\begin{question}
Is the property of being Markov preserved under passing to
subsemigroups and extensions of finite Grigorchuk index for finitely
generated cancellative semigroups (so that both of the semigroups are
finitely generated)?
\end{question}

\section{The class of Markov languages}
\label{sec:markovlang}

This final section examines the class of languages that are Markov
languages for some semigroup or monoid. First, notice that not every
regular language is a Markov language:

\begin{example}
\label{ex:nonmarkovlang}
Let $L = a^+ \cup a^+b$. Suppose $L$ is a Markov language for a
semigroup $S$. Then $\elt{b}$ lies in $S$ and so must be represented
by an element of $L$. If $\elt{b} = \elt{a^k}$ for some $k$ then
$\elt{ab} = \elt{aa^k} = \elt{a^{k+1}}$. Since both $ab$ and $a^{k+1}$
lie in $L$, this contradicts the uniqueness of representives in
$L$. If, on the other hand, $\elt{b} = \elt{a^kb}$ for some $k$, then
$\elt{ab} = \elt{aa^kb} = \elt{a^{k+1}b}$, again contradicting the
uniqueness of representives in $L$. So $L$ is not a semigroup Markov
language.

Indeed, if instead $L' = L \cup \{\emptyword\} = a^* \cup a^+b$, then
the same contradictions show that $L'$ is not a monoid Markov
language.
\end{example}

Starting from a Markov language and adding or removing a finite number
of words can yield a prefix-closed regular language that is not a
Markov language, as the following two examples show:

\begin{example}
Let $K = L' \cup \{b\} = a^* \cup a^*b$, where $L'$ is the language
from \fullref{Example}{ex:nonmarkovlang}. Then $K$ is a Markov language for the
semigroup presented by $\pres{a,b}{(b^2,b),(ba,b)}$. To see this,
notice that $(\{a,b\},\{(b^2,b),(ba,b)\})$ is a confluent noetherian
rewriting system and its language of normal forms is $K$, and apply
\fullref{Proposition}{prop:rsmarkov}. Thus removing the single word $b$
from the Markov language $K$ yields the non-Markov language $L'$.
\end{example}

\begin{example}
Let $L = a^* \cup \{a^2c,a^4c\}$. Suppose $L$ is a Markov language for
a semigroup $S$. Then $\elt{ac}$ lies in $S$ and so must be
represented by an element of $L$. Now, if $\elt{ac} = \elt{a^\alpha}$,
then $\elt{a^2c} = \elt{a^{\alpha+1}}$, contradicting the uniqueness of
representatives in $L$. If $\elt{ac} = \elt{a^2c}$, then $\elt{ac} =
\elt{a^2c} = \elt{a^3c} = \elt{a^4c}$, again contradicting the
uniqueness of representatives in $L$. So $\elt{ac} = \elt{a^4c}$.

Now, $\elt{a^3c}$ must also be represented by an element of $L$. If
$\elt{a^3c} = \elt{a^\alpha}$, then $\elt{a^4c} = \elt{a^{\alpha+1}}$,
  contradicting the uniqueness of representatives in $L$. If
  $\elt{a^3c} = \elt{a^2c}$, then $\elt{a^2c} = \elt{a^3c} =
  \elt{a^4c}$, again contradicting the uniqueness of representatives
  in $L$. So $\elt{a^3c} = \elt{a^4c}$, which, by the preceding
  paragraph, implies $\elt{ac} = \elt{a^3c}$, which in turn implies
  $\elt{a^2c} = \elt{a^4c}$. This contradicts the uniqueness of
  representatives in $L$, and so $L$ cannot be a Markov language.

Thus adding the two words $a^2c$ and $a^4c$ to the Markov language
$a^*$ yields the non-Markov language $L$.
\end{example}

There are two main questions about the class of Markov languages:

\begin{question}
Is there an algorithm that takes a regular language that is prefix-closed
or $+$-prefix-closed and decides whether it is
a Markov language for some monoid or semigroup?
\end{question}

\begin{question}
Is every finite language that is prefix-closed or $+$-prefix-closed a
Markov language for a (necessarily finite) monoid or semigroup?
\end{question}



\bibliography{\jobname,automaticsemigroups,automaticpresentations,languages,presentations,semigroups,c_publications}
\bibliographystyle{alphaabbrv}

\end{document}